\documentclass[12pt]{amsart}
\usepackage{}
\usepackage{amsmath}
\usepackage{amsfonts}
\usepackage{amssymb}
\usepackage[all]{xy}           

\usepackage{bbding}
\usepackage{txfonts}
\usepackage{amscd}
\usepackage{xspace}

\usepackage[shortlabels]{enumitem}
\usepackage{ifpdf}
\ifpdf
\usepackage[colorlinks,final,backref=page,hyperindex]{hyperref}
\else
\usepackage[colorlinks,final,backref=page,hyperindex,hypertex]{hyperref}
\fi
\usepackage{tikz}
\usepackage[active]{srcltx}
\usepackage{amsmath,amssymb,mathrsfs,mathtools,microtype}
\makeatletter
\newtheorem{thm}{Theorem}[section]

\newtheorem{pro}[thm]{Proposition}
\theoremstyle{definition}
\newtheorem{defi}[thm]{Definition}

\newtheorem{rmk}[thm]{Remark}

\newcommand {\emptycomment}[1]{}

\newcommand{\tr}[1]{\textcolor{blue}{#1}}
\newcommand{\TR}[1]{\textcolor{purple}{#1}}

\newcommand{\postLie}{\mathsf{PL}}

\newcommand{\nc}{\newcommand}
\nc{\delete}[1]{{}}
\nc{\CV}{\mathbf{C}}

\nc{\oprn}{\theta}
\nc{\Oprn}{\Theta}

\newcommand{\lon }{\,\rightarrow\,}
\newcommand{\be }{\begin{equation}}
	\newcommand{\ee }{\end{equation}}

\nc{\frakg}{\mathfrak{g}}

\newcommand{\g}{\mathfrak g}
\newcommand{\h}{\mathfrak h}

\newcommand{\Labn}{\mathfrak{n}}
	\newcommand{\Labe}{\mathfrak{e}}

\nc{\adj}{\xspace}
\nc{\opt}{operator\xspace}
\nc{\Opt}{Operator\xspace}

\nc{\calo}{\mathcal{O}}

\newcommand{\huaH}{\mathcal{H}}

\newcommand{\PLie}{\mathsf{pre}}

\newcommand{\s}{\mathfrak{s}}

\newcommand{\frkC}{\mathfrak C}

\newcommand{\half}{\frac{1}{2}}

\newcommand{\Id}{{\rm{Id}}}

\newcommand{\br}[1]{   [ \cdot,    \cdot  ]   }

\newcommand{\Hom}{\mathrm{Hom}}

\newcommand{\Der}{\mathsf{Der}}

\tikzset{
	cross/.style={path picture={
			\draw[symbols]
			(path picture bounding box.south east) -- (path picture bounding box.north west) (path picture bounding box.south west) -- (path picture bounding box.north east);
	}},
	root/.style={circle,fill=green!50!black,inner sep=0pt, minimum size=1.2mm},
	dot/.style={circle,fill=pageforeground,inner sep=0pt, minimum size=1mm},
	dotred/.style={circle,fill=pageforeground!50!pagebackground,inner sep=0pt, minimum size=2mm},
	var/.style={circle,fill=pageforeground!10!pagebackground,draw=pageforeground,inner sep=0pt, minimum size=3mm},
	kernel/.style={semithick,shorten >=2pt,shorten <=2pt},
	kernels/.style={snake=zigzag,shorten >=2pt,shorten <=2pt,segment amplitude=1pt,segment length=4pt,line before snake=2pt,line after snake=5pt,},
	rho/.style={densely dashed,semithick,shorten >=2pt,shorten <=2pt},
	testfcn/.style={dotted,semithick,shorten >=2pt,shorten <=2pt},
	renorm/.style={shape=circle,fill=pagebackground,inner sep=1pt},
	labl/.style={shape=rectangle,fill=pagebackground,inner sep=1pt},
	xic/.style={very thin,circle,draw=symbols,fill=symbols,inner sep=0pt,minimum size=1.2mm},
	g/.style={very thin,rectangle,draw=symbols,fill=symbols!10!pagebackground,inner sep=0pt,minimum width=2.5mm,minimum height=1.2mm},
	xi/.style={very thin,circle,draw=symbols,fill=symbols!10!pagebackground,inner sep=0pt,minimum size=1.2mm},
	xies/.style={very thin,rectangle,fill=green!50!black!25,draw=symbols,inner sep=0pt,minimum size=1.1mm},
	xiesf/.style={very thin,rectangle,fill=green!50!black,draw=symbols,inner sep=0pt,minimum size=1.1mm},
	xix/.style={very thin,crosscircle,fill=symbols!10!pagebackground,draw=symbols,inner sep=0pt,minimum size=1.2mm},
	X/.style={very thin,cross,rectangle,fill=pagebackground,draw=symbols,inner sep=0pt,minimum size=1.2mm},
	xib/.style={thin,circle,fill=symbols!10!pagebackground,draw=symbols,inner sep=0pt,minimum size=1.6mm},
	xie/.style={thin,circle,fill=green!50!black,draw=symbols,inner sep=0pt,minimum size=1.6mm},
	xid/.style={thin,circle,fill=symbols,draw=symbols,inner sep=0pt,minimum size=1.6mm},
	xibx/.style={thin,crosscircle,fill=symbols!10!pagebackground,draw=symbols,inner sep=0pt,minimum size=1.6mm},
	kernels2/.style={very thick,draw=connection,segment length=12pt},
	keps/.style={thin,draw=symbols,->},
	kepspr/.style={thick,draw=connection,->},
	krho/.style={thin,draw=symbols,superdense,->},
	krhopr/.style={thick,draw=connection,superdense},
	triangle/.style = { regular polygon, regular polygon sides=3},
	not/.style={thin,circle,draw=connection,fill=connection,inner sep=0pt,minimum size=0.5mm},
	diff/.style = {very thin,draw=symbols,triangle,fill=red!50!black,inner sep=0pt,minimum size=1.6mm},
	diff1/.style = {very thin,dectriangle={1}{0},fill=red!50!black,draw=symbols,inner sep=0pt,minimum size=1.6mm},
	diff2/.style = {very thin,dectriangle={1}{1},fill=red!50!black,draw=symbols,inner sep=0pt,minimum size=1.6mm},
	diffmini/.style = {very thin,rectangle,fill=black,draw=black,inner sep=0pt,minimum size=0.75mm},
	kernelsmod/.style={very thick,draw=connection,segment length=12pt},
	rec/.style = {very thin,rectangle,fill=black,draw=black,inner sep=0pt,minimum size=2mm},
	cerc/.style={very thin,circle,draw=black,fill=symbols,inner sep=0pt,minimum size=2mm},
	stars/.style={very thin,star,star points=6,star point ratio=0.5, draw=black,fill=red,inner sep=0pt,minimum size=0.7mm},
	>=stealth,
}
\tikzset{
	root/.style={circle,fill=black!50,inner sep=0pt, minimum size=3mm},
	circ/.style={circle,fill=white,draw=black,very thin,inner sep=.5pt, minimum size=1.2mm},
	round1/.style={fill=white,outer sep = 0,inner sep=2pt,rounded corners=1mm,draw,text=black,thin,minimum size=1.2mm},
	circ1/.style={circle,fill=red!10,draw=red,very thin,inner sep=.5pt, minimum size=1.2mm},
	rect/.style={fill=white,outer sep = 0,inner sep=2pt,rectangle,draw,text=black,thin,minimum size=1.2mm},
	rect1/.style={fill=white,outer sep = 0,inner sep=2pt,rectangle,draw,text=black,thin,minimum size=1.2mm},
	round2/.style={fill=red!10,outer sep = 0,inner sep=2pt,rounded corners=1mm,draw,text=black,thin,minimum size=1.2mm},
	round3/.style={fill=blue!10,outer sep = 0,inner sep=2pt,rounded corners=1mm,draw,text=black,thin,minimum size=1.2mm},
	rect2/.style={fill=black!10,outer sep = 0,inner sep=2pt,rectangle,draw,text=black,thin,minimum size=1.2mm},
	dot/.style={circle,fill=black,inner sep=0pt, minimum size=1.2mm},
	dotred/.style={circle,fill=black!50,inner sep=0pt, minimum size=2mm},
	var/.style={circle,fill=black!10,draw=black,inner sep=0pt, minimum size=3mm},
	kernel/.style={semithick,shorten >=2pt,shorten <=2pt},
	diag/.style={thin,shorten >=4pt,shorten <=4pt},
	kernel1/.style={thick},
	kernels/.style={snake=zigzag,shorten >=2pt,shorten <=2pt,segment amplitude=1pt,segment length=4pt,line before snake=2pt,line after snake=5pt,},
	kernels1/.style={snake=zigzag,segment amplitude=0.5pt,segment length=2pt},
	rho1/.style={densely dotted,semithick},
	rho/.style={densely dashed,semithick,shorten >=2pt,shorten <=2pt},
	testfcn/.style={dotted,semithick,shorten >=2pt,shorten <=2pt},
	visible/.style={draw, circle, fill, inner sep=0.25ex},
	renorm/.style={shape=circle,fill=white,inner sep=1pt},
	labl/.style={shape=rectangle,fill=white,inner sep=1pt},
	xic/.style={very thin,circle,fill=symbols,draw=black,inner sep=0pt,minimum size=1.2mm},
	xi/.style={very thin,circle,fill=blue!10,draw=black,inner sep=0pt,minimum size=1.2mm},
	xib/.style={very thin,circle,fill=blue!10,draw=black,inner sep=0pt,minimum size=1.6mm},
	xie/.style={very thin,circle,fill=green!50!black,draw=black,inner sep=0pt,minimum size=1mm},
	xid/.style={very thin,circle,fill=symbols,draw=black,inner sep=0pt,minimum size=1.6mm},
	edgetype/.style={very thin,circle,draw=black,inner sep=0pt,minimum size=5mm},
	nodetype/.style={very thick,circle,draw=black,inner sep=0pt,minimum size=5mm},
	kernels2/.style={very thick,draw=connection,segment length=12pt},
	clean/.style={thin,circle,fill=black,inner sep=0pt,minimum size=1mm},	not/.style={thin,circle,fill=symbols,draw=connection,fill=connection,inner sep=0pt,minimum size=0.8mm},
	>=stealth,
}

\begin{document}
	
	\title[Post-Lie deformations of pre-Lie algebras]{Post-Lie deformations of pre-Lie algebras  and their applications in Regularity Structures}

	\author{Yvain Bruned}
	\address{Universit\'e de Lorraine, CNRS, IECL, F-54000 Nancy, France}
	\email{yvain.bruned@univ-lorraine.fr}
	
	\author{Yunhe Sheng}
	\address{Department of Mathematics, Jilin University, Changchun 130012, Jilin, China}
	\email{shengyh@jlu.edu.cn}
	
	\author{Rong Tang}
	\address{Department of Mathematics, Jilin University, Changchun 130012, Jilin, China}
	\email{tangrong@jlu.edu.cn}
	
	
	
	\begin{abstract}
In this paper, we study post-Lie deformations of a pre-Lie algebra, namely deforming a pre-Lie algebra into a post-Lie algebra. We construct the differential graded Lie algebra that governs post-Lie deformations of a pre-Lie algebra. We also develop the post-Lie cohomology theory for a pre-Lie algebra, by which we classify infinitesimal post-Lie deformations of a pre-Lie algebra using the second cohomology group. The rigidity of such kind of deformations is also characterized using the second cohomology group. Finally, we apply this deformation theory to Regularity Structures. We prove that  the post-Lie algebraic structure on the decorated trees which appears spontaneously in  Regularity Structures is a post-Lie deformation of a pre-Lie algebra.
		
	\end{abstract}
	
\renewcommand{\thefootnote}{}
	
	
	\maketitle
	\tableofcontents

	\allowdisplaybreaks
	

\section{Introduction}

The notion of a post-Lie algebra has been introduced by Vallette in the course of study of Koszul duality of operads \cite{Val}.
Munthe-Kaas and his coauthors found that post-Lie algebras also naturally appear in differential geometry \cite{GMS} and numerical integration on  manifolds \cite{Munthe-Kaas-Lundervold}.  Meanwhile, it was found that post-Lie algebras play an essential role  in Regularity Structures in stochastic analysis \cite{BHZ,BK,Hairer}.
Recently, post-Lie algebras have been studied from different points of view including constructions of nonabelian generalized Lax pairs \cite{BGN},    Poincar\'e-Birkhoff-Witt type theorems \cite{D,Gubarev}, factorization theorems \cite{EMM}, multiple zeta values \cite{Burmester-1,Burmester-2} and relations to post-Lie groups \cite{AFM,BGST,MQS}. 

A pre-Lie algebra is a post-Lie algebra whose underlying Lie algebra is abelian. Pre-Lie
algebras are a class of nonassociative algebras coming from the study of convex homogeneous cones, affine manifolds and affine
structures on Lie groups, and  cohomologies of associative algebras.  They also appeared in many fields in
mathematics and mathematical physics, such as complex and symplectic structures on Lie groups and Lie algebras, integrable
systems,  vertex algebras and operads. See \cite{Bai,Bu,Ma,Sm22b} for more details. The homotopy theory of pre-Lie algebras was given in \cite{CL} using the operadic approach. Furthermore, the cohomology theory of pre-Lie algebras was given by Dzhumadil'daev and Boyom in \cite{DA} and \cite{Boyom}.

The method of deformation is ubiquitous in mathematics and physics.
Roughly speaking, a deformation of a mathematical structure is a perturbation which gives the same  kind of structure. Motivated by the foundational  work of Kodaira and Spencer~\cite{KS} for complex analytic structures, deformation theory finds its generalization in algebraic geometry~\cite{Ha} and further in number theory as deformations of Galois representations~\cite{Maz}. The deformation of algebraic structures began with the seminal
work of Gerstenhaber~\cite{Ge0,Ge} for associative
algebras and followed by its extension to Lie algebras by
Nijenhuis and Richardson~\cite{NR}. Deformations of other
algebraic structures such as pre-Lie algebras have also been
developed~\cite{Bu0}. In general, deformation theory was developed
for binary quadratic operads by Balavoine~\cite{Bal}.

In this paper, we consider the problem of deforming pre-Lie algebras into post-Lie algebras. On the one hand, we construct the differential graded Lie algebra that governs post-Lie deformations of a pre-Lie algebra. On the other hand, we develop the post-Lie cohomology theory for a pre-Lie algebra, such that the second cohomology group classifies infinitesimal post-Lie deformations of a pre-Lie algebra. We also show that if the second post-Lie cohomology group of a pre-Lie algebra is trivial, then the pre-Lie algebra is post-Lie rigid.

Then we apply the established deformation theory to Regularity Structures.  In Hairer's seminal work \cite{Hairer} in stochastic partial differential equations, he introduced the notion of  Regularity Structures which  provides an algebraic framework  to describe the local Taylor expansion of functions around each point. One can think that various kinds of the theory of rough paths and  the usual Taylor expansions have been  unified into Regularity Structures. It is noteworthy that decorated rooted trees and related combinatorial Hopf algebras play pivotal role in the algebraic renormalisation of Regularity Structures \cite{BHZ}. Recently, the natural post-Lie algebraic structures on the combining objects  as decorated rooted trees, multi-indices and noncommutative derivations have been discovered and simplified the combinatorial Hopf algebraic structures in Regularity Structures \cite{BK,JZ}. We prove that  the post-Lie algebraic structure on the decorated trees is a post-Lie deformation of a pre-Lie algebra. Moreover, take advantage of the Taylor deformation of the post-Lie algebraic structure on the decorated trees, we construct a concrete formal post-Lie deformation of the pre-Lie algebra on the decorated trees.

The paper is organized as follows. In Section \ref{sec:con}, we give the differential graded Lie algebra that governs post-Lie deformations of a pre-Lie algebra. In Section \ref{sec:coh}, we introduce the post-Lie cohomology for a pre-Lie algebra, by which we study formal post-Lie deformations of a pre-Lie algebra. In Section \ref{sec:app}, we apply the deformation theory described above to the post-Lie algebra appearing in the context of Regularity Structures.

	\section{Post-Lie deformations of  pre-Lie algebras}\label{sec:con}
In this section, we study post-Lie deformations of a pre-Lie algebra, and construct the differential graded Lie algebra that governs this kind of deformations.

	\begin{defi} (\cite{Val})\label{post-lie-defi}
		A {\bf post-Lie algebra} $(\g,[\cdot,\cdot],\rhd)$ consists of a Lie algebra $(\g,[\cdot,\cdot])$ and a binary product $\rhd:\g\otimes\g\lon\g$ such that
		\begin{eqnarray}
			\label{Post-1}x\rhd[y,z]&=&[x\rhd y,z]+[y,x\rhd z],\\
			\label{Post-2}([x,y]+x\rhd y-y\rhd x) \rhd z&=&x\rhd(y\rhd z)-y\rhd(x\rhd z),\,\,\forall x, y, z\in \g.
		\end{eqnarray}
	\end{defi}

		Any post-Lie algebra $(\g,[\cdot,\cdot],\rhd)$ gives rise to a new Lie algebra
		$\g_\rhd:=(\g,[\cdot,\cdot]_{\g_\rhd})$
		defined by
		$$[x,y]_{\g_\rhd} \coloneqq x\rhd y-y\rhd x+[x,y],\quad\forall x,y\in\g,$$
		which is called the {\bf sub-adjacent Lie algebra}. Moreover,
		Eqs.~\eqref{Post-1}-\eqref{Post-2} equivalently mean that the linear map $L:\g\to\mathfrak g\mathfrak l(\g)$ defined by $L_{x}y=x\rhd y$ is an action of the Lie algebra $(\g,[\cdot,\cdot]_{\g_\rhd})$ on the Lie algebra $(\g,[\cdot,\cdot])$.
	
	\begin{rmk}
		Let $(\g,[\cdot,\cdot],\rhd)$ be a post-Lie algebra. If the Lie bracket $[\cdot,\cdot]$ vanishes, then $(\g,\rhd)$ becomes a pre-Lie algebra. For further details on
pre-Lie algebras, see \cite{Bai,Bu,Ma}. Thus,  a post-Lie algebra can be viewed as a nonabelian generalization of a pre-Lie algebra.
	\end{rmk}

\begin{defi}\label{post-lie-homo-defi}
		\begin{itemize}
			\item  A {\bf derivation} of a post-Lie algebra $(\g,[\cdot,\cdot],\rhd)$ is a linear map $d:\g\lon\g$  satisfying $
				d([x,y])=[d(x),y]+[x,d(y)]$ and $
				d(x\rhd y)=d(x)\rhd y+x\rhd d(y)$ for all $x,y\in\g.$
			Denote by $\Der(\g)$ the set of derivations of the post-Lie algebra $(\g,[\cdot,\cdot],\rhd)$.
			
			\item   A {\bf homomorphism} from a post-Lie algebra $(\g,[\cdot,\cdot],\rhd)$ to a post-Lie algebra $(\h,[\cdot,\cdot],\rhd)$ is a Lie algebra morphism $\phi:\g\lon \h$ such that
			$
				\phi(x\rhd y)=\phi(x)\rhd \phi(y)$ for all $x,y\in\g.
		$
			In particular, if $\phi$ is  invertible,  then $\phi$ is called an  {\bf isomorphism}.
		\end{itemize}
		
	\end{defi}

	 We consider the following question:

\begin{center}
{\bf  Deforming a pre-Lie algebra into a post-Lie algebra. }
\end{center}

 Consider the post-Lie deformation  of a pre-Lie algebra, we have the following characterization.

\begin{thm}\label{formula-post-deform}
Let $(\g,\rhd)$ be a pre-Lie algebra. Let $\pi\in\Hom(\wedge^2\g,\g)$ and $\omega\in\Hom(\g\otimes\g,\g)$ be linear maps. Then $(\g,\pi,\rhd+\omega)$ is a post-Lie algebra if and only if
\begin{itemize}
  \item[\rm{(i)}] $(\g,\pi)$ is a Lie algebra;

  \item[\rm{(ii)}] 	$x\rhd\pi(y,z)+\omega(x,\pi(y,z))=\pi(x\rhd y,z)+\pi(y,x\rhd z)+\pi(\omega(x,y),z)+\pi(y,\omega(x,z));
	$

  \item[\rm{(iii)}] 	
  $
		  \big(\omega(x,y)-\omega(y,x)+\pi(x,y)\big)\rhd z+\omega\big(x\rhd y-y\rhd x,z\big)+\omega(\omega(x,y)-\omega(y,x)+\pi(x,y), z)\\
		=\omega(x,y\rhd z)+x\rhd \omega(y,z)-\omega(y,x\rhd z)-y\rhd \omega_1(x,z)+\omega(x,\omega(y, z))-\omega(y,\omega(x, z)).
$
\end{itemize}
In this case,  we say that $(\g,\pi,\rhd+\omega)$
is a {\bf  post-Lie deformation} of the pre-Lie algebra $(\g,\rhd)$.
\end{thm}

	In the sequel, we give an intrinsic explanation of the above compatibility conditions, and give the differential graded Lie algebra that governs post-Lie deformations of a pre-Lie algebra. First we recall the controlling algebra of post-Lie algebras. We will write $\mathbb S_n$ for the group of permutations on $n$ symbols. A permutation $\sigma\in\mathbb S_n$ is called an $(i,n-i)$-shuffle if $\sigma(1)<\ldots <\sigma(i)$ and $\sigma(i+1)<\ldots <\sigma(n)$. If $i=0$ or $n$, we assume $\sigma=\Id$. The set of all $(i,n-i)$-shuffles will be denoted by $\mathbb S_{(i,n-i)}$. The notion of an $(i_1,\ldots,i_k)$-shuffle and the set $\mathbb S_{(i_1,\ldots,i_k)}$ are defined analogously.

	Let $V$ be a vector space. Consider the graded vector space
	$
	C_{\postLie}^*(V,V)=\oplus_{n=0}^{+\infty}C_{\postLie}^{n}(V,V),
	$
	 given by the formula
	$$
	C_{\postLie}^{n}(V,V)=\oplus_{i=0}^{n}\Hom(\wedge^{i}V\otimes \wedge^{n+1-i}V,V).
	$$
	We will write $f=(f_0,\ldots,f_n)\in C_{\postLie}^{n}(V,V)$, where $f_i\in \Hom(\wedge^iV\otimes\wedge^{n+1-i}V,V)$. For $f\in C_{\postLie}^{n}(V,V)$ and $g\in C_{\postLie}^{m}(V,V)$, define  $f\circ g=\big((f\circ g)_0,\ldots,(f\circ g)_{n+m}\big)\in C^{n+m}_\postLie(V,V)$ for $0\le k\le m$ by
	\begin{eqnarray*}
		&&(f\circ g)_k(x_1 \ldots x_{k}\otimes x_{k+1} \ldots x_{n+m+1})\\
		&=&\sum_{\sigma\in\mathbb S_{(k-j,j)}\atop 0\le j\le k}\sum_{ \tau\in\mathbb S_{(m+1-j,n+j-k)}}(-1)^\sigma(-1)^\tau(-1)^{m(k-j)}\\
		\nonumber&&f_{k-j}\big(x_{\sigma(1)} \ldots  x_{\sigma(k-j)}\otimes  g_{j}(x_{\sigma(k-j+1)} \ldots x_{\sigma(k)}\otimes x_{k+\tau(1)} \ldots x_{k+\tau(m+1-j)}) x_{k+\tau(m+2-j)} \ldots  x_{k+\tau(n+m+1-k)}\big),
	\end{eqnarray*}
	and for $m+1\le k\le m+n$ by
	\begin{eqnarray*}
		\nonumber&&(f\circ g)_{k}(x_1 \ldots  x_k\otimes x_{k+1} \ldots  x_{n+m+1})\\
		&=&\sum_{\sigma\in \mathbb S_{(j,m+1-j,k-m-1)}\atop 0\le j\le m}(-1)^{\sigma}f_{k-m}\big(g _{j}(x_{\sigma(1)} \ldots x_{\sigma(j)}\otimes x_{\sigma(j+1)} \ldots x_{\sigma(m+1)}) x_{\sigma(m+2)} \ldots  x_{\sigma(k)}\otimes  x_{k+1} \ldots x_{n+m+1}\big)\\
		\nonumber&&+\sum_{\sigma\in\mathbb S_{(k-j,j)}\atop 0\le j\le m}\sum_{ \tau\in\mathbb S_{(m+1-j,n+j-k)}}(-1)^\sigma(-1)^\tau(-1)^{m(k-j)}\\
		\nonumber&&f_{k-j}\big(x_{\sigma(1)} \ldots  x_{\sigma(k-j)}\otimes  g_{j}(x_{\sigma(k-j+1)} \ldots x_{\sigma(k)}\otimes x_{k+\tau(1)} \ldots x_{k+\tau(m+1-j)}) x_{k+\tau(m+2-j)} \ldots  x_{k+\tau(n+m+1-k)}\big).
	\end{eqnarray*}

	\begin{thm}\cite{LST24}\label{post-lie-gla}
		Let $V$ be a vector space. Then  $(C_{\postLie}^*(V,V),[\cdot,\cdot]_{\postLie}) $ is a graded Lie algebra, where the graded Lie bracket $[\cdot,\cdot]_{\postLie}$ is given by
		\begin{equation}
			[f,g]_{\postLie}:=f\circ g-(-1)^{nm}g\circ f,
			\quad \forall f\in C_{\postLie}^{n}(V,V),~g\in C_{\postLie}^{m}(V,V).
			\label{eq:glapL}
		\end{equation}
		Moreover, its Maurer-Cartan elements are precisely post-Lie algebra structures on the vector space  $V$.
	\end{thm}

	Let $(\g, \rhd)$ be a pre-Lie algebra.   By Theorem \ref{post-lie-gla},   $\rhd\in \Hom(\g\otimes\g,\g)$ is a Maurer-Cartan element of the graded Lie algebra   $(C_{\postLie}^*(\g,\g),[\cdot,\cdot]_{\postLie})$.
	Define the inner derivation $d_\rhd:C_{\postLie}^*(\g,\g)\lon C_{\postLie}^*(\g,\g)$ by
	\begin{eqnarray*}
		d_\rhd(f):=[\rhd,f]_{\postLie},\,\,\,\,\forall f\in C_{\postLie}^*(\g,\g).
	\end{eqnarray*}
	Since a pre-Lie algebra is a special post-Lie algebra, so we have  $[\rhd,\rhd]_{\postLie}=0$ and we obtain a
	differential graded Lie algebra $(C_{\postLie}^*(\g,\g),[\cdot,\cdot]_{\postLie},d_\rhd)$. This differential graded Lie algebra governs post-Lie deformations of the pre-Lie algebra $(\g,\rhd)$.

	\begin{thm}\label{def-post-lie}
		Let $(\g,\rhd)$ be a pre-Lie algebra. Then   $ (\pi,\omega)$ is  a  post-Lie deformation of the pre-Lie algebra $(\g,\rhd)$, where $\pi:\g\wedge\g\longrightarrow \g$ and $\omega:\g\otimes\g\longrightarrow \g$ are linear maps,   if and only if $\Pi=(\pi,\omega)$ is a Maurer-Cartan
		element of the differential graded Lie algebra
		$(C_{\postLie}^*(\g,\g),[\cdot,\cdot]_{\postLie},d_\rhd)$.
	\end{thm}
	\begin{proof}
	By Theorem \ref{post-lie-gla},	  $ (\pi,\omega)$ is  a  post-Lie deformation of the pre-Lie algebra $(\g,\rhd)$ if and only if
$$
[(\pi,\rhd+\omega),(\pi,\rhd+\omega)]_{\postLie}=0.
$$
By the fact $[\rhd,\rhd]_{\postLie}=0$, this is equivalent to
$$
2[\rhd,\Pi]_{\postLie}+[\Pi,\Pi]_{\postLie}=0.
$$
Therefore, we have
$$
d_\rhd\Pi+\half[\Pi,\Pi]_{\postLie}=0,
$$
which implies that $\Pi=(\pi,\omega)$ is a Maurer-Cartan element.
	\end{proof}
	
	\begin{rmk}
The condition that  $\Pi=(\pi,\omega)$ is a Maurer-Cartan
		element of the differential graded Lie algebra
		$(C_{\postLie}^*(\g,\g),[\cdot,\cdot]_{\postLie},d_\rhd)$ is equivalent to the conditions given in Theorem \ref{formula-post-deform}.
\end{rmk}
	
	\section{Post-Lie cohomologies and  formal post-Lie deformations of pre-Lie algebras}\label{sec:coh}

	In this section, we define the post-Lie cohomology of a  pre-Lie algebra. As an application, we classify infinitesimal post-Lie deformations of a pre-Lie algebra using the second cohomology group of the relevant complex.
	
	Let $(\g,\rhd)$ be a pre-Lie algebra. Define the set of  $0$-cochains $\frkC^0_{\postLie}(\g;\g)$ to be $0$, and define
	the set of  $n$-cochains $\frkC^n_{\postLie}(\g;\g)$ by
	$$
	\frkC^n_{\postLie}(\g;\g)=\oplus_{i=0}^{n-1}\Hom(\wedge^{i}\g\otimes \wedge^{n-i}\g,\g).
	$$
		Define the {\bf coboundary operator} $\partial:\frkC^n_{\postLie}(\g;\g)\lon \frkC^{n+1}_{\postLie}(\g;\g)$ by
 \begin{equation}\label{Coboundary-O}
 \partial =\sum_{k=0}^{n-2}(\partial^k_{k+1}+\partial^k_n)+\partial^{n-1}_n,
 \end{equation}
 where $\partial^k_j$, $j\in\{k+1,n\}$ is the map $\Hom(\wedge^k\g\otimes \wedge^{n-k}\g,\g)\to\Hom(\wedge^j\g\otimes \wedge^{n+1-j}\g,\g)$ given by the following formulas:
	\begin{itemize}
				\item for $k=0,1,\ldots,n-2$,
		\begin{eqnarray*}
			&&(\partial_{k+1}^{k}f_{k})(x_1,\ldots, x_{k+1}; x_{k+2}, \ldots, x_{n+1})\\
			&=&\sum_{i=1}^{k+1}(-1)^{i-1}x_i\rhd f_{k}(x_1,\ldots, \hat{x}_i,\ldots, x_{k+1}; x_{k+2}, \ldots, x_{n+1})\\
			&&+\sum_{1\le i< j\le k+1}(-1)^{i+j}f_{k}\big((x_i\rhd x_j-x_j\rhd x_i), x_1,\ldots, \hat{x}_i,\ldots, \hat{x}_j,\ldots, x_{k+1}; x_{k+2}, \ldots, x_{n+1}\big)\\
			&&-\sum_{i=1}^{k+1}\sum_{j=k+2}^{n+1}(-1)^{i+j+k+1}f_{k}(x_1, \ldots,\hat{x}_i,\ldots, x_{k+1}; x_i\rhd x_j, x_{k+2},\ldots, \hat{x}_j,\ldots, x_{n+1}),
		\end{eqnarray*}

		\item for $k=0,1,\ldots,n-2$,
				\begin{eqnarray*}
			(\partial_{n}^{k} f_k)(x_1, \ldots,  x_{n}; x_{n+1})=\sum_{\sigma\in \mathbb S_{(k,n-k)}}(-1)^{n-1}(-1)^{\sigma}f_k(x_{\sigma(1)}, \ldots, x_{\sigma(k)}; x_{\sigma(k+1)}, \ldots, x_{\sigma(n)})\rhd x_{n+1},
		\end{eqnarray*}
	\end{itemize}
	and
\begin{eqnarray*}
		&&(\partial_{n}^{n-1} f_{n-1})(x_1, \ldots,  x_{n}; x_{n+1})\\
		&=&\sum_{i=1}^{n}(-1)^{i+1}x_i\rhd f_{n-1}(x_1, \ldots,\hat{x}_i,\ldots,  x_{n}; x_{n+1})\\
		&&+\sum_{i=1}^{n}(-1)^{i+1}f_{n-1}(x_1,\ldots,\hat{x}_i,\ldots, x_{n};x_i)\rhd x_{n+1}\\
		&&+\sum_{1\le i<j\le n}(-1)^{i+j}f_{n-1}\big((x_i\rhd x_j-x_j\rhd x_i), x_{1},\ldots, \hat{x}_i,\ldots, \hat{x}_j, \ldots, x_{n}; x_{n+1}\big)\\
		&&-\sum_{i=1}^{n}(-1)^{i+1}f_{n-1}(x_{1},\ldots, \hat{x}_i,\ldots,  x_{n}; x_i\rhd x_{n+1}).
	\end{eqnarray*}
We introduce the notations $C^{k,s}=\Hom(\wedge^k\g\otimes\wedge^{s}\g,\g)$.
This is	illustrated by the following diagram:
{\footnotesize
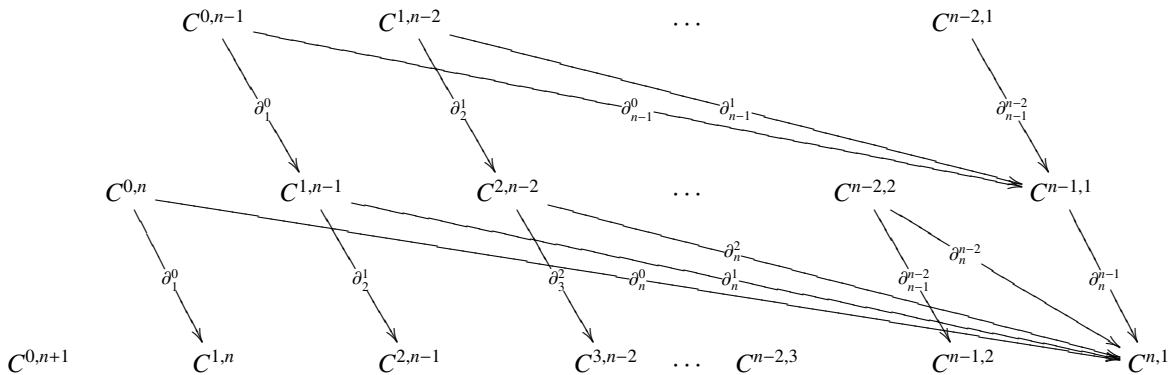
\begin{equation*}
\xymatrix@C=7pt@R=50pt{
&&C^{0,n-1}
  \ar[dr]|{{\partial}^0_1} \ar[drrrrrrrrr]|{{\partial}^0_{n-1}} &
&C^{1,n-2}
  \ar[dr]|{{\partial}^1_2} \ar[drrrrrrr]|{{\partial}^1_{n-1}}&
&&\cdots&
&&C^{n-2,1}
  \ar[dr]| {{\partial}^{n-2}_{n-1}}&
&&\\
&C^{0,n}
  \ar[dr]|{{\partial}^0_1} \ar[drrrrrrrrrrr]|{{\partial}^0_n}&
&C^{1,n-1}
 \ar[dr]|{{\partial}^1_2} \ar[drrrrrrrrr]|{{\partial}^1_n}&
&C^{2,n-2}
  \ar[dr]|{{\partial}^2_3} \ar[drrrrrrr]|<<<<<<<<<<<<<<<<<<<<<<<<<{{\partial}^2_n}&
&\cdots&
&C^{n-2,2}
 \ar[dr]|{{\partial}^{n-2}_{n-1}} \ar[drrr]|<<<<<<<<<<<<{{\partial}^{n-2}_{n}}&
&C^{n-1,1}
 \ar[dr]|{{\partial}^{n-1}_{n}}&\\
C^{0,n+1}& & C^{1,n}& & C^{2,n-1}& & C^{3,n-2}& \cdots& C^{n-2,3}& & C^{n-1,2}& &C^{n,1}
}
\end{equation*}
}

\begin{thm}
 With the above notations,  $(\oplus _{n=0}^{+\infty}\frkC^n_{\postLie}(\g;\g),\partial)$ is a cochain complex, i.e. $\partial\circ \partial=0$.
\end{thm}
\begin{proof}
A pre-Lie algebra is a special post-Lie algebra, and it is straightforward to see that the above $\partial$ is exactly the  coboundary operator for post-Lie algebras given in \cite{LST24} applying to the particular case of pre-Lie algebras.
\end{proof}
	
	\begin{defi}\label{defi:cohomology of pL}
		The cohomology of the cochain complex $(\oplus _{n=0}^{+\infty}\frkC^n_{\postLie}(\g;\g),\partial)$ is called the {\bf post-Lie cohomology  of the pre-Lie algebra} $(\g,\rhd)$. We denote its $n$-th cohomology group by $\huaH^n_{\postLie}(\g;\g)$.
	\end{defi}

	\begin{rmk}
		Note that $(\oplus_{n=1}^{+\infty}\Hom(\wedge^{n-1} \g\otimes \g,\g), \partial^{n-1}_n)$ is a subcomplex, whose cohomology is exactly the usual cohomology of the pre-Lie algebra $(\g,\rhd)$. We denote this complex by $\frkC^*_{\PLie}(\g;\g)$ and its $n$-th cohomology group by $\huaH^n_{\PLie}(\g;\g)$. See \cite{Boyom,DA} for details about this cohomology and its applications. Thus, there is a natural embedding of the usual cohomology group into the post-Lie cohomology group of a pre-Lie algebra.
	\end{rmk}

\begin{rmk}
  It is straightforward to obtain that $(\oplus_{n=1}^{+\infty}\Hom(\wedge^{n-1} \g\otimes \wedge^k\g,\g), \partial^{*-1}_*)$ is a cochain complex for all $k=2,3,\cdots.$ We denote this complex by $\frkC^*_{\PLie_k}(\g;\g)$ and its $n$-th cohomology group by $\huaH^n_{\PLie_k}(\g;\g)$,~ $k=2,3,\cdots.$
\end{rmk}
	
\begin{thm}\label{cohomology-exact}
Let $(\g,\rhd)$ be a pre-Lie algebra. Then there is a short exact sequence of cochain complexes:
\begin{eqnarray}\label{pre-lie-exact}
0\longrightarrow\frkC^*_{\PLie}(\g;\g)\stackrel{\iota}{\longrightarrow}\frkC^*_{\postLie}(\g;\g)\stackrel{p}{\longrightarrow} \oplus_{k=2}^{+\infty}\frkC^*_{\PLie_k}(\g;\g)\longrightarrow 0,
\end{eqnarray}
where $\iota$ and $p$ are the inclusion map and the projection map.

Consequently, there is a long exact sequence of  cohomology groups:
$$
\cdots\longrightarrow\huaH^n_{\PLie}(\g;\g)\stackrel{\huaH^n(\iota)}{\longrightarrow}\huaH^n_{\postLie}(\g;\g)\stackrel{\huaH^n(p)}{\longrightarrow} \oplus_{k=2}^{n}\huaH^n_{\PLie_k}(\g;\g) \stackrel{c^n}\longrightarrow \huaH^{n+1}_{\PLie}(\g;\g)\longrightarrow\cdots,
$$
where the connecting map $c^n$ is defined by
$
c^n([f_0],\cdots,[f_{n-2}])=[\sum_{k=0}^{n-2}\partial^k_nf_k],$  for all $[f_k]\in \huaH^n_{\PLie_{n-k}}(\g;\g).$
\end{thm}
\begin{proof}
By \eqref{Coboundary-O}, we obtain the above short exact sequence of   cochain complexes. Moreover, by diagram chasing   the connecting map $c^n$ is given by $
c^n([f_0],\cdots,[f_{n-2}])=[\sum_{k=0}^{n-2}\partial^k_nf_k]$. The proof is finished.
\end{proof}
Next we give applications of the first and the second cohomology group.

Note that $f\in \frkC^1_{\postLie}(\g;\g)=\Hom(\g,\g)$  is closed  if and only if
	\begin{eqnarray}
				\partial(f)_1(x_1\otimes x_2)&=&f(x_1)\rhd x_2+x_1\rhd f(x_2)-f(x_1\rhd x_2)=0.
	\end{eqnarray}
	Therefore, we have the following proposition.
	
	\begin{pro}
		A $1$-cochain  $ f\in \frkC^1_{\postLie}(\g;\g)$ is a $1$-cocycle if and only if $f$ is a derivation. Moreover,  $$\huaH^1(\g;\g)=\Der(\g).$$
	\end{pro}

	Note that $\frkC^2_{\postLie}(\g;\g)=\Hom(\wedge^2\g,\g)\oplus \Hom(\otimes^2\g,\g)$.  Let $\pi\in\Hom(\wedge^2\g,\g)$ and $\omega\in\Hom(\otimes^2\g,\g)$ be  linear maps, i.e. $(\pi,\omega)\in\frkC^2_{\postLie}(\g;\g)$. Moreover, $(\pi,\omega)\in \frkC^2_{\postLie}(\g;\g)$  is closed  if and only if
\begin{eqnarray}
		\label{2-cocycle-1}\partial(\pi,\omega)_1(x;y,z)&=&x\rhd\pi(y,z)-\pi(x\rhd y,z)+\pi(x\rhd z,y)=0,\\
		\label{2-cocycle-2}\partial(\pi,\omega)_2(x,y;z)&=&-\big(\omega(x,y)-\omega(y,x)+\pi(x,y)\big)\rhd z-\omega\big(x\rhd y-y\rhd x,z\big)\\
		\nonumber&&+\omega(x,y\rhd z)+x\rhd \omega(y,z)-\omega(y,x\rhd z)-y\rhd \omega(x,z)=0.
	\end{eqnarray}
Let $\mathbb{R}[[t]]$ be the ring of power series in one variable
$t$. For any $\mathbb{R}$-linear space $\g$, we let
$\g[[t]]$ denote the vector space of formal power series in $t$ with coefficients in $\g$. Let $(\g,\rhd)$ be a pre-Lie algebra.
We consider the power series
\begin{eqnarray}
\pi_t&=&\sum_{i=0}^{+\infty}\pi_i t^i,\quad \pi_i\in \Hom_{\mathbb{R}}(\wedge^{2}\g,\g),\\
\omega_t&=&\sum_{i=0}^{+\infty}\omega_i t^i,\quad\omega_i\in \Hom_{\mathbb{R}}(\otimes^2\g,\g).
\label{eq:tdeform}
\end{eqnarray}

\begin{defi}
Let $(\g,\rhd)$ be a pre-Lie algebra. If $(\g[[t]],\pi_t,\omega_t)$ is a $\mathbb{R}[[t]]$-post-Lie algebra with $(\pi_0,\omega_0)=(0,\rhd)$, we say that $(\g[[t]],\pi_t,\omega_t)$
is a {\bf  formal post-Lie deformation} of the pre-Lie algebra $(\g,\rhd)$.
\end{defi}

A pair $(\pi_t,\omega_t)$, as given above, is a formal post-Lie deformation of a pre-Lie algebra $(\g,\rhd)$ if and only if for all $x,y,z\in \g$, the following equalities   hold:
\begin{eqnarray}
\label{deformation1}&&\pi_t(\pi_t(x,y),z)+\pi_t(\pi_t(y,z),x)+\pi_t(\pi_t(z,x),y)=0,\\
\label{deformation2}&&\omega_t(x,\pi_t(y,z))=\pi_t(\omega_t(x,y),z)+\pi_t(y,\omega_t(x,z)),\\
\label{deformation3}&&\omega_t(\omega_t(x,y)-\omega_t(y,x)+\pi_t(x,y),z)=\omega_t(x,\omega_t(y,z))-\omega_t(y,\omega_t(x,z)).
\end{eqnarray}
 Expanding the equations in \eqref{deformation1}-\eqref{deformation3} and collecting coefficients of $t^n$, we see that  \eqref{deformation1}-\eqref{deformation3} are equivalent to the system of equations
\begin{eqnarray}
\label{deformation4}&&\sum\limits_{i+j=n\atop i,j\geq0}\Big(\pi_i(\pi_j(x,y),z)+\pi_i(\pi_j(y,z),x)+\pi_i(\pi_j(z,x),y)\Big)=0,\\
\label{deformation5}&&\sum\limits_{i+j=n\atop i,j\geq0}\Big(\omega_i(x,\pi_j(y,z))-\pi_j(\omega_i(x,y),z)-\pi_j(y,\omega_i(x,z))\Big)=0,\\
\label{deformation6}&&\sum\limits_{i+j=n\atop i,j\geq0}\Big(\omega_i(\omega_j(x,y)-\omega_j(y,x)+\pi_j(x,y),z)-\omega_i(x,\omega_j(y,z))+\omega_i(y,\omega_j(x,z))\Big)=0.
\end{eqnarray}

\begin{rmk}\label{rmk}
For $n=2$, by $\pi_0=0$, condition \eqref{deformation4} is equivalent to the Jacobi identity of $\pi_1$.
 \end{rmk}

 \begin{pro}\label{pro:cocycle}
 Let $(\pi_t,\omega_t)$ be a formal post-Lie deformation of a pre-Lie algebra $(\g,\rhd)$.
Then $(\pi_1,\omega_1)\in \frkC^2_{\postLie}(\g;\g)$   is a $2$-cocycle in the post-Lie cohomology  of the pre-Lie algebra $(\g,\rhd)$.
\end{pro}
\begin{proof}
For $n=1$, \eqref{deformation5} is equivalent to
$$
x\rhd\pi_1(y,z)-\pi_1(x\rhd y,z)-\pi_1(y,x\rhd z)=0,
$$
and \eqref{deformation6} is equivalent to
\begin{eqnarray*}
&&(\omega_1(x,y)-\omega_1(y,x)+\pi_1(x,y))\rhd z+\omega_1(x\rhd y-y\rhd x,z)\\
&&-\omega_1(x,y\rhd z)-x\rhd\omega_1(y,z)+\omega_1(y,x\rhd z)+y\rhd\omega_1(x,z)=0.
\end{eqnarray*}
By \eqref{2-cocycle-1} and \eqref{2-cocycle-2}, we deduce that $(\pi_1,\omega_1)$ is a $2$-cocycle. The proof is finished.
\end{proof}

\begin{defi}
Let $(\pi_t,\omega_t)$ be a formal post-Lie deformation of a pre-Lie algebra $(\g,\rhd)$. The $2$-cocycle $(\pi_1,\omega_1)$ is called the {\bf infinitesimal} of the  formal post-Lie deformation $(\pi_t,\omega_t)$ of the pre-Lie algebra $(\g,\rhd)$.
\end{defi}
	
\begin{defi}
Let $(\bar{\pi}_t,\bar{\omega}_t)$ and $(\pi_t,\omega_t)$ be  formal post-Lie deformations of a pre-Lie algebra $(\g,\rhd)$. A
{\bf formal isomorphism} from $(\bar{\pi}_t,\bar{\omega}_t)$ to $(\pi_t,\omega_t)$ is a power series
$$
\Phi_t=\sum_{i\geq0}\phi_it^i:\g[[t]]\lon\g[[t]],
$$
where  $\phi_i\in \Hom_{\mathbb{R}}(\g,\g)$ with $\phi_0=\Id$, such that
\begin{eqnarray}
&&\Phi_t\circ\bar{\pi}_t= \pi_t\circ(\Phi_t\times\Phi_t),\\
&&\Phi_t\circ\bar{\omega}_t=\omega_t\circ(\Phi_t\times\Phi_t).
\end{eqnarray}
Two formal post-Lie deformations $(\bar{\pi}_t,\bar{\omega}_t)$ and $(\pi_t,\omega_t)$  are said to be {\bf equivalent} if  there exists a formal isomorphism $\Phi_t$ from $(\bar{\pi}_t,\bar{\omega}_t)$ to $(\pi_t,\omega_t)$.
\end{defi}

\begin{thm}
The infinitesimals of two equivalent formal post-Lie deformations of a pre-Lie algebra $(\g,\rhd)$ are in the same cohomology class.
\end{thm}
\begin{proof}
Let $\Phi_t:(\bar{\pi}_t,\bar{\omega}_t)\lon(\pi_t,\omega_t)$ be a formal isomorphism. Then for all $x,y\in \g$, we have
\begin{eqnarray*}
\Phi_t\circ\bar{\pi}_t(x,y)&=& \pi_t\circ(\Phi_t\times\Phi_t)(x,y),\\
\Phi_t\circ\bar{\omega}_t(x,y)&=& \omega_t\circ(\Phi_t\times\Phi_t)(x,y).
\end{eqnarray*}
 Expanding the above identities and comparing coefficients of $t$, we have
\begin{eqnarray*}
\bar{\pi}_1(x,y)&=&\pi_1(x,y),\\
\bar{\omega}_1(x,y)&=&\omega_1(x,y)+\phi_1(x)\rhd y+x\rhd \phi_1(y)-\phi_1(x\rhd y).
\end{eqnarray*}
Thus, we have $(\bar{\pi}_1,\bar{\omega}_1)=(\pi_1,\omega_1)+\partial(\phi_1)$, which implies that $[(\bar{\pi}_1,\bar{\omega}_1)]=[(\pi_1,\omega_1)]\in\huaH^2_{\postLie}(\g;\g)$. The proof is finished.
\end{proof}

Let $(\g,\rhd)$ be a pre-Lie algebra. In the sequel, we will also denote the binary product $\rhd$ by $\omega$.
\begin{defi}
A formal post-Lie deformation $(\pi_t,\omega_t)$ of a  pre-Lie algebra $(\g,\omega)$ is said to be {\bf trivial} if it is equivalent to $(0,\omega)$, i.e. there exists  $\Phi_t=\sum_{i\geq0}\phi_it^i:\g[[t]]\lon\g[[t]]$, where  $\phi_i\in  \Hom_{\mathbb{R}}(\g,\g)$ with $\phi_0=\Id$, such that
\begin{eqnarray}
&&\Phi_t\circ\bar{\pi}_t= 0,\\
&&\Phi_t\circ\bar{\omega}_t= \omega\circ(\Phi_t\times\Phi_t).
\end{eqnarray}
 \end{defi}
	
\begin{defi}
A pre-Lie algebra $(\g,\omega)$ is said to be {\bf post-Lie rigid} if  every  formal post-Lie deformation of $(\g,\omega)$ is trivial.
\end{defi}

\begin{thm}
If $\huaH^2_{\postLie}(\g;\g)=0$, then the pre-Lie algebra $(\g,\omega)$ is post-Lie rigid.
\end{thm}

\begin{proof}
 Let $(\pi_t,\omega_t)$ be a  formal post-Lie deformation of a pre-Lie algebra $(\g,\omega)$. By Proposition \ref{pro:cocycle}, we deduce that $(\pi_1,\omega_1)$ is a $2$-cocycle. By $\huaH^2_{\postLie}(\g;\g)=0,$ there exists a 1-cochain $\phi_1\in  \frkC^1_{\postLie}(\g;\g)$ such that
\begin{eqnarray}
\label{rigid}(\pi_1,\omega_1)=-\partial(\phi_1)=(0,-\omega\circ(\phi_1\times \Id)-\omega\circ(\Id\times \phi_1)+\phi_1\circ \omega).
\end{eqnarray}
 Then setting $\Phi_t={\Id}+\phi_1 t$, we have a deformation $(\bar{\pi}_t,\bar{\omega}_t)$, where
\begin{eqnarray*}
\bar{\pi}_t(x,y)&=&\big(\Phi_t^{-1}\circ \pi_t\circ(\Phi_t\times\Phi_t)\big)(x,y),\\
\bar{\omega}_t(x,y)&=&\big(\Phi_t^{-1}\circ \omega_t\circ(\Phi_t\times\Phi_t)\big)(x,y).
\end{eqnarray*}
Thus, $(\bar{\pi}_t,\bar{\omega}_t)$ is equivalent to $(\pi_t,\omega_t)$. Moreover, we have
\begin{eqnarray*}
\bar{\pi}_t(x,y)&=&({\Id}-\phi_1t+\phi_1^2t^{2}+\cdots+(-1)^i\phi_1^it^{i}+\cdots)(\pi_t(x+\phi_1(x)t,y+\phi_1(y)t)),\\
\bar{\omega}_t(x,y)&=&({\Id}-\phi_1t+\phi_1^2t^{2}+\cdots+(-1)^i\phi_1^it^{i}+\cdots)(\omega_t(x+\phi_1(x)t,y+\phi_1(y)t)).
\end{eqnarray*}
Thus, we have
\begin{eqnarray*}
\bar{\pi}_t(x,y)&=&\bar{\pi}_{2}(x,y)t^{2}+\cdots,\\
\bar{\omega}_t(x,y)&=&\omega(x,y)+(\omega_1(x,y)+\omega(x,\phi_1(y))+\omega(\phi_1(x),y)-\phi_1(\omega(x,y)))t+\bar{\omega}_{2}(x,y)t^{2}+\cdots.
\end{eqnarray*}
By \eqref{rigid}, we have
\begin{eqnarray*}
\bar{\pi}_t(x,y)&=&\bar{\pi}_{2}(x,y)t^{2}+\cdots,\\
\bar{\omega}_t(x,y)&=&\omega(x,y)+\bar{\omega}_{2}(x,y)t^{2}+\cdots.
\end{eqnarray*}
Then by repeating the argument, we can show that $(\pi_t,\omega_t)$ is equivalent to $(0,\omega)$.
\end{proof}

\section{Post-Lie deformations of pre-Lie algebras in Regularity Structures}\label{sec:app}

In this section, we want to apply the formulism described above to the post-Lie algebra appearing in the context of Regularity Structures \cite{BK,JZ}.

Decorated trees as introduced in
\cite{BHZ} are described in the following way. We suppose given two symbols $I$ and $\Xi$ and let $ \mathcal{D} := \lbrace I,\Xi \rbrace \times \mathbb{N}^{d+1}$ define the set of edge decorations. These two symbols represent a convolution with a kernel $ K $ and a noise term $ \xi $ that correspond to the mild formulation of a singular stochastic partial differential equation of the form:
\[
u = K* \left( f(\partial^{\alpha} u, \alpha \in \mathbb{R}^{d+1}) + g(\partial^{\alpha} u, \alpha \in \mathbb{R}^{d+1}) \xi  \right)
\]
where $ * $ is space-time convolution and $f$ and $g$ are non-linearities depending on a finite number of $ \partial^{\alpha} u, \alpha \in \mathbb{R}^{d+1}$.
For a system and more noises, one has to add more symbols.
 Decorated trees over $ \mathcal{D} $ are of the form  $T_{\Labe}^{\Labn} =  (T,\Labn,\Labe) $ where $T$ is a non-planar rooted tree with node set $N_T$ and edge set $E_T$. The maps $\Labn : N_T \rightarrow \mathbb{N}^{d+1}$ and $\Labe : E_T \rightarrow \mathcal{D}$ are node, respectively edge, decorations. We denote the set of decorated trees by $ \mathfrak{T} $.
We introduce a symbolic notation for denoting these decorated trees:
\begin{enumerate}
	\item[--] An edge decorated by  $ (I,a) \in \mathcal{D} $  is denoted by $ I_{a} $. The symbol $  I_{a} $ is also viewed as  the operation that grafts a tree onto a new root via a new edge with edge decoration $(I,a) $. The new root at hand remains decorated with $0$.

	\item[--]  An edge decorated by $ (\Xi,0) \in \mathcal{D} $ is denoted by $  \Xi $.  We suppose that these edges are terminal edges. A terminal edge of a tree $T$ is an edge with one extremity being a leaf of $T$. We suppose also that these edges have  zero node decoration at their leaves that it appears at most once for each node.
	 Let us stress that we do not allow edges of the form $(\Xi,b),~b\neq 0$ and that an edge which is not a terminal edge must be decorated by $(I,a)$.
	
	\item[--] A factor $ X^{\ell}$   encodes a single node  $ \bullet^{\ell} $ decorated by $ \ell \in \mathbb{N}^{d+1}$. We write $ X_i$, $ i \in \lbrace 0,1,\ldots,d\rbrace $, to denote $ X^{e_i}$. Here, we have denoted by $ e_i $ the vector of $ \mathbb{N}^{d+1} $ with $ 1 $ in $i$th position and $ 0 $ otherwise. The element $ X^0 $ is identified with the empty tree $\mathbf{1}$.
\end{enumerate}

Given a decorated tree $ \tau $ there exist decorated trees $ \tau_1, ..., \tau_r $ such that
\begin{equation} \label{decomposition decorated trees}
	\tau = X^{\ell} \Xi^m  \prod_{i=1}^r I_{a_i}(\tau_i),
\end{equation}
where $m \in \lbrace 0,1 \rbrace$, $\ell, a_i \in \mathbb{N}^{d+1}$ and the notation means that $ \tau $ is formed of a root decorated by $ \ell $. The edges connected to this root are: an edge decorated by $ (\Xi,0) $, edges decorated by $(I, a_i)$ with the other extremities being the roots of the $\tau_i$.
A tree of the form $ I_a(\tau) $ is called a planted tree as there is only one edge connecting the root to the rest of the tree.
The decomposition \eqref{decomposition decorated trees} can be viewed as a unique product decomposition for the tree product as  $X^{\ell}$ is the decorated tree
$\bullet^{\ell}$ and $\Xi$ is the decorated tree with one edge decorated by $(\Xi,0)$.

Below, we present an examples of  decorated trees:
\begin{equation*}
	\tau = X^{\alpha} \Xi I_a(X^{\beta})  =   \begin{tikzpicture}[scale=0.2,baseline=0.1cm]
		\node at (0,0)  [dot,label= {[label distance=-0.2em]below: \scriptsize  $  \alpha   $} ] (root) {};
		\node at (2,4)  [dot,label={[label distance=-0.2em]above: \scriptsize  $ \beta $}] (right) {};
		\node at (-2,4)  [dot,label={[label distance=-0.2em]above: \scriptsize  $ $} ] (left) {};
		\draw[kernel1] (right) to
		node [sloped,below] {\small }     (root); \draw[kernel1] (left) to
		node [sloped,below] {\small }     (root);
		\node at (-1,2) [fill=white,label={[label distance=0em]center: \scriptsize  $ \Xi $} ] () {};
		\node at (1,2) [fill=white,label={[label distance=0em]center: \scriptsize  $ I_a $} ] () {};
	\end{tikzpicture}, \quad
	I_b(\tau) =
	\begin{tikzpicture}[scale=0.2,baseline=0.1cm]
		\node at (0,0)  [dot,label= {[label distance=-0.2em]below: \scriptsize  $  $} ] (root) {};
		\node at (-3,3)  [dot,label={[label distance=-0.2em]left: \scriptsize  $ \alpha $} ] (left) {};
		\node at (0,7)  [dot,label={[label distance=-0.2em]above: \scriptsize  $ \beta $} ] (center) {};
		\node at (-6,7)  [dot,label={[label distance=-0.2em]above: \scriptsize  $  $} ] (centerr) {};
		\draw[kernel1] (left) to
		node [sloped,below] {\small }     (root);
		\draw[kernel1] (center) to
		node [sloped,below] {\small }     (left);
		\draw[kernel1] (centerr) to
		node [sloped,below] {\small }     (left);
		\node at (-1.5,1.5) [fill=white,label={[label distance=0em]center: \scriptsize  $ I_b $} ] () {};
		\node at (-1.5,5) [fill=white,label={[label distance=0em]center: \scriptsize  $ I_a $} ] () {};
		\node at (-4.5,5) [fill=white,label={[label distance=0em]center: \scriptsize  $ \Xi $} ] () {};
	\end{tikzpicture}
\end{equation*}

We define a product called grafting product:
\begin{equation*}
	\sigma \curvearrowright^a \tau:=\sum_{v\in  N_{\tau} } \sigma \curvearrowright^a_v  \tau,
\end{equation*}
where $\sigma $ and $\tau$ are two decorated rooted trees, $ N_\tau $ is the set of vertices of $ \tau $ and where $\sigma \curvearrowright^a_v \tau$ is obtained by grafting the tree $\sigma$ on the tree $\tau$ at vertex $v$ by means of a new edge decorated by $a\in \mathbb{N}^{d+1}$. Grafting onto noise-type edges, that is, edges decorated by $ (\Xi,0) $ is forbidden. Below, we provide an example of this grafting product:
\begin{equation}
	\label{example_1}
	\bullet^{\alpha}  \curvearrowright^a   \begin{tikzpicture}[scale=0.2,baseline=0.1cm]
		\node at (0,0)  [dot,label= {[label distance=-0.2em]below: \scriptsize  $  \gamma   $} ] (root) {};
		\node at (2,4)  [dot,label={[label distance=-0.2em]above: \scriptsize  $ \beta $}] (right) {};
		\node at (-2,4)  [dot,label={[label distance=-0.2em]above: \scriptsize  $ $} ] (left) {};
		\draw[kernel1] (right) to
		node [sloped,below] {\small }     (root); \draw[kernel1] (left) to
		node [sloped,below] {\small }     (root);
		\node at (-1,2) [fill=white,label={[label distance=0em]center: \scriptsize  $ \Xi $} ] () {};
		\node at (1,2) [fill=white,label={[label distance=0em]center: \scriptsize  $ I_b $} ] () {};
	\end{tikzpicture} = \begin{tikzpicture}[scale=0.2,baseline=0.1cm]
		\node at (0,0)  [dot,label= {[label distance=-0.2em]below: \scriptsize  $  \gamma   $} ] (root) {};
		\node at (0,5)  [dot,label= {[label distance=-0.2em]above: \scriptsize  $  \alpha   $} ] (center) {};
		\node at (3,4)  [dot,label={[label distance=-0.2em]above: \scriptsize  $ \beta $}] (right) {};
		\node at (-3,4)  [dot,label={[label distance=-0.2em]above: \scriptsize  $ $} ] (left) {};
		\draw[kernel1] (right) to
		node [sloped,below] {\small }     (root);
		\draw[kernel1] (center) to
		node [sloped,below] {\small }     (root);
		\draw[kernel1] (left) to
		node [sloped,below] {\small }     (root);
		\node at (-2,2) [fill=white,label={[label distance=0em]center: \scriptsize  $ \Xi $} ] () {};
		\node at (2,2) [fill=white,label={[label distance=0em]center: \scriptsize  $ I_b $} ] () {};
		\node at (0,2.5) [fill=white,label={[label distance=0em]center: \scriptsize  $ I_a $} ] () {};
	\end{tikzpicture} + \begin{tikzpicture}[scale=0.2,baseline=0.1cm]
		\node at (0,0)  [dot,label= {[label distance=-0.2em]below: \scriptsize  $  \gamma  $} ] (root) {};
		\node at (0,8)  [dot,label= {[label distance=-0.2em]above: \scriptsize  $  \alpha   $} ] (center) {};
		\node at (2,4)  [dot,label={[label distance=-0.2em]right: \scriptsize  $ \beta $}] (right) {};
		\node at (-2,4)  [dot,label={[label distance=-0.2em]above: \scriptsize  $ $} ] (left) {};
		\draw[kernel1] (right) to
		node [sloped,below] {\small }     (root);
		\draw[kernel1] (center) to
		node [sloped,below] {\small }     (right); \draw[kernel1] (left) to
		node [sloped,below] {\small }     (root);
		\node at (-1,2) [fill=white,label={[label distance=0em]center: \scriptsize  $ \Xi $} ] () {};
		\node at (1,2) [fill=white,label={[label distance=0em]center: \scriptsize  $ I_b $} ] () {};
		\node at (1,6) [fill=white,label={[label distance=0em]center: \scriptsize  $ I_a $} ] () {};
	\end{tikzpicture}.
\end{equation}
 The family of grafting products $ (\curvearrowright^b  )_{b \in \mathbb{N}^{d+1}} $ forms a multi-pre-Lie algebra structure on the decorated rooted trees, in the sense that they satisfy the following identities:
\begin{equation*}
	\left( \tau_1  \curvearrowright^a \tau_2 \right)  \curvearrowright^b \tau_3 - 	 \tau_1  \curvearrowright^a (  \tau_2  \curvearrowright^b \tau_3 ) = 	( \tau_2  \curvearrowright^b \tau_1 )  \curvearrowright^a \tau_3 - 	 \tau_2  \curvearrowright^b (  \tau_1  \curvearrowright^a \tau_3 )
\end{equation*}
where the $ \tau_i $ are decorated trees and $ a,b $ belong to $ \mathbb{N}^{d+1} $.
It has first been introduced in \cite[Prop. 4.21]{BCCH} for writing up renormalised equations.
This multi-pre-Lie algebra can be summarised into a single pre-Lie structure on the space of planted trees, given by the product:
\begin{equation*}
	I_{a}(\sigma) \curvearrowright I_{b}(\tau):= I_{b}(\sigma \curvearrowright^a \tau).
\end{equation*}
This was first noticed in \cite{F2018} (see also \cite[Prop. 3.2]{BM22}).

Furthermore, the products $ \curvearrowright^a $ can be deformed via a multi-pre-Lie algebra isomorphism described in \cite[Sec. 2.2]{BM22}. The deformed products, first introduced in \cite{BCCH}, are given by:
\begin{equation*}
	\sigma \widehat{\curvearrowright}^a \tau:=\sum_{v\in N_{\tau}}\sum_{\ell\in\mathbb{N}^{d+1}}{\Labn_v \choose \ell} \sigma  \curvearrowright_v^{a-\ell}(\uparrow_v^{-\ell} \tau),
\end{equation*}
\label{deformed_grafting_a}
where $ \Labn_v \in \mathbb{N}^{d+1}$ denotes  the decoration at the vertex $ v $ and the operator $ \uparrow_v^{-\ell} $ is defined  as adding $ - \ell $ to the node decoration of $ v $.
If there exists a unique pair $(b,\alpha)\in \mathbb{N}^{d+1} \times \mathbb{N}^{d+1}$ such that $a=\ell+b$ and $\Labn_v =\ell+\alpha$ then one can compute the decorations of the new decorated tree. If this condition is not satisfied, then the decorated tree is set to be zero. Given a scaling $\s \in\mathbb{N}_0^{d+1} = \mathbb{N}^{d+1} \setminus \lbrace 0 \rbrace$ we define the \textsl{grading} of a tree as the sum of the gradings of its edges and denote it by $ |\cdot|_{\text{grad}} $:
\begin{equation*}
	|\tau|_{\text{grad}}:=\sum_{e\in E_{\tau}}\big|\Labe(e) \big|_{\s}
\end{equation*}
where $ E_{\tau} $ are the edges of $ \tau $, $ \Labe(e) = (\Labe(e)_1,\Labe(e)_2) \in \mathcal{D} $ is the decoration of the edge $ e $ and $\big|\Labe(e) \big|_{\s}$ is defined as $\big|\Labe(e)_2 \big|_{\s}$. For a given $ \mathbf{n} \in \mathbb{N}^{d+1} $, one has:
\begin{equation*}
	|\mathbf n|_{\s}:= \sum_{i=0}^d s_i n_i.
\end{equation*}
The scaling $\s$ comes from the anisotropic Euclidean norm used in the context of singular SPDEs. In practice, when the kernel  $K$ is given as the inverse of the operator $ \partial_t - \Delta $, the first coordinate of $\mathbb{R}^{d+1}$ is the time and counts double in comparison to the other  spatial components ($\s  = (2,1,...,1)$). It is called the parabolic scaling.
One notices that $\widehat{\curvearrowright}^a$ is a deformation of $\curvearrowright^a$ in the sense that:
\begin{equation*}
	\sigma \widehat{\curvearrowright}^a \tau = \sigma \curvearrowright^a \tau +
	\hbox{ lower grading terms}.
\end{equation*}
One may summarise the above family of multi-pre-Lie products into a single pre-Lie product $\widehat{\curvearrowright}$:
\begin{equation*}
	I_{a}(\sigma) \, \widehat{\curvearrowright} \, I_{b}(\tau):= I_{b}(\sigma \, \widehat{\curvearrowright}^{a} \, \tau).
\end{equation*}
Below, we provide an example illustrating the product $ \widehat{\curvearrowright} $:
\begin{equation*}
	\begin{tikzpicture}[scale=0.2,baseline=0.1cm]
		\node at (0,0)  [dot,label= {[label distance=-0.2em]below: \scriptsize  $    $} ] (root) {};
		\node at (0,4)  [dot,label={[label distance=-0.2em]above: \scriptsize  $ \alpha $}] (center) {};
		\draw[kernel1] (center) to
		node [sloped,below] {\small }     (root);
		\node at (0,2) [fill=white,label={[label distance=0em]center: \scriptsize  $ I_a $} ] () {};
	\end{tikzpicture} 	 \widehat{\curvearrowright} \begin{tikzpicture}[scale=0.2,baseline=0.1cm]
		\node at (0,0)  [dot,label= {[label distance=-0.2em]below: \scriptsize  $     $} ] (root) {};
		\node at (0,8)  [dot,label= {[label distance=-0.2em]above: \scriptsize  $    $} ] (center) {};
		\node at (2,4)  [dot,label={[label distance=-0.2em]right: \scriptsize  $ \beta $}] (right) {};
		\draw[kernel1] (right) to
		node [sloped,below] {\small }     (root);
		\draw[kernel1] (center) to
		node [sloped,below] {\small }     (right);
		\node at (1,2) [fill=white,label={[label distance=0em]center: \scriptsize  $ I_b $} ] () {};
		\node at (1,6) [fill=white,label={[label distance=0em]center: \scriptsize  $ \Xi $} ] () {};
	\end{tikzpicture}   =   \begin{tikzpicture}[scale=0.2,baseline=0.1cm]
		\node at (0,0)  [dot,label= {[label distance=-0.2em]below: \scriptsize  $     $} ] (root) {};
		\node at (0,8)  [dot,label= {[label distance=-0.2em]above: \scriptsize  $    $} ] (center) {};
		\node at (4,8)  [dot,label= {[label distance=-0.2em]above: \scriptsize  $ \alpha   $} ] (centerl) {};
		\node at (2,4)  [dot,label={[label distance=-0.2em]right: \scriptsize  $ \beta $}] (right) {};
		\draw[kernel1] (right) to
		node [sloped,below] {\small }     (root);
		\draw[kernel1] (center) to
		node [sloped,below] {\small }     (right);
		\draw[kernel1] (centerl) to
		node [sloped,below] {\small }     (right);
		\node at (1,2) [fill=white,label={[label distance=0em]center: \scriptsize  $ I_b $} ] () {};
		\node at (3,6) [fill=white,label={[label distance=0em]center: \scriptsize  $ I_a $} ] () {};
		\node at (1,6) [fill=white,label={[label distance=0em]center: \scriptsize  $ \Xi $} ] () {};
	\end{tikzpicture} + \sum_{\ell \in \mathbb{N}^{d+1}_0 }{\beta \choose \ell} \, \, \begin{tikzpicture}[scale=0.2,baseline=0.1cm]
		\node at (0,0)  [dot,label= {[label distance=-0.2em]below: \scriptsize  $     $} ] (root) {};
		\node at (0,8)  [dot,label= {[label distance=-0.2em]above: \scriptsize  $    $} ] (center) {};
		\node at (4,8)  [dot,label= {[label distance=-0.2em]above: \scriptsize  $ \alpha   $} ] (centerl) {};
		\node at (2,4)  [dot,label={[label distance=-0.2em]right: \scriptsize  $ \beta - \ell $}] (right) {};
		\draw[kernel1] (right) to
		node [sloped,below] {\small }     (root);
		\draw[kernel1] (center) to
		node [sloped,below] {\small }     (right);
		\draw[kernel1] (centerl) to
		node [sloped,below] {\small }     (right);
		\node at (1,2) [fill=white,label={[label distance=0em]center: \scriptsize  $ I_b $} ] () {};
		\node at (3,6) [fill=white,label={[label distance=0em]center: \scriptsize  $ \quad I_{a - \ell}  $} ] () {};
		\node at (1,6) [fill=white,label={[label distance=0em]center: \scriptsize  $ \Xi $} ] () {};
	\end{tikzpicture}
\end{equation*}
Another important operation we will need to define on the space of trees is $ \uparrow^{i} $:
\begin{equation*}
	\uparrow^{i} \tau  = \sum_{v \in N_{\tau}} \uparrow^{e_i}_v \tau,
\end{equation*}
with the convention that $ \uparrow^{e_i}_v \tau= 0 $ if $v$ is a leaf which is part of an edge decorated by $(\Xi,0)$.
 We provide below an example of computation:
\begin{equation*}
	\uparrow^{i} \begin{tikzpicture}[scale=0.2,baseline=0.1cm]
		\node at (0,0)  [dot,label= {[label distance=-0.2em]below: \scriptsize  $  \gamma  $} ] (root) {};
		\node at (2,4)  [dot,label={[label distance=-0.2em]above: \scriptsize  $ \beta $}] (right) {};
		\node at (-2,4)  [dot,label={[label distance=-0.2em]above: \scriptsize  $ $} ] (left) {};
		\draw[kernel1] (right) to
		node [sloped,below] {\small }     (root); \draw[kernel1] (left) to
		node [sloped,below] {\small }     (root);
		\node at (-1,2) [fill=white,label={[label distance=0em]center: \scriptsize  $ \Xi $} ] () {};
		\node at (1,2) [fill=white,label={[label distance=0em]center: \scriptsize  $ I_b $} ] () {};
	\end{tikzpicture} =  \begin{tikzpicture}[scale=0.2,baseline=0.1cm]
		\node at (0,0)  [dot,label= {[label distance=-0.2em]below: \scriptsize  $  \gamma + e_i  $} ] (root) {};
		\node at (2,4)  [dot,label={[label distance=-0.2em]above: \scriptsize  $ \beta $}] (right) {};
		\node at (-2,4)  [dot,label={[label distance=-0.2em]above: \scriptsize  $ $} ] (left) {};
		\draw[kernel1] (right) to
		node [sloped,below] {\small }     (root); \draw[kernel1] (left) to
		node [sloped,below] {\small }     (root);
		\node at (-1,2) [fill=white,label={[label distance=0em]center: \scriptsize  $ \Xi $} ] () {};
		\node at (1,2) [fill=white,label={[label distance=0em]center: \scriptsize  $ I_b $} ] () {};
	\end{tikzpicture} +  \begin{tikzpicture}[scale=0.2,baseline=0.1cm]
		\node at (0,0)  [dot,label= {[label distance=-0.2em]below: \scriptsize  $  \gamma  $} ] (root) {};
		\node at (2,4)  [dot,label={[label distance=-0.2em]above: \scriptsize  $ \beta + e_i $}] (right) {};
		\node at (-2,4)  [dot,label={[label distance=-0.2em]above: \scriptsize  $ $} ] (left) {};
		\draw[kernel1] (right) to
		node [sloped,below] {\small }     (root); \draw[kernel1] (left) to
		node [sloped,below] {\small }     (root);
		\node at (-1,2) [fill=white,label={[label distance=0em]center: \scriptsize  $ \Xi $} ] () {};
		\node at (1,2) [fill=white,label={[label distance=0em]center: \scriptsize  $ I_b $} ] () {};
	\end{tikzpicture}.
\end{equation*}

This operator is a derivation for the grafting  product $ \curvearrowright^a $ in the sense that
\begin{equation*} \label{derivation_i}
	\uparrow^{i} \left( \sigma \curvearrowright^a \tau  \right) =  (\uparrow^{i} \sigma) \curvearrowright^a  \tau +  \sigma \curvearrowright^a \, ( \uparrow^{i} \tau).
\end{equation*}
We define the following spaces:
\begin{equation*}
	\begin{aligned}
		\mathcal{V} & = \Big \langle  \{ I_a(\tau), \, a \in \mathbb{N}^{d+1}, \, \tau \in \mathfrak{T} \} \cup \{ X_{i} \}_{i = 0,..., d} \Big \rangle_{\mathbb{R}}, \\
		\tilde{\mathcal{V}} & = \Big  \langle \{ I_a(\tau), \, a \in \mathbb{N}^{d+1}, \, \tau \in \mathfrak{T} \} \Big \rangle_{\mathbb{R}}.
	\end{aligned}
\end{equation*}
The  space $ \tilde{\mathcal{V}} $ is the linear span of planted decorated trees and $ \mathcal{V}$ is the linear span of planted trees with the monomials $ X_i $. In the next definition, we recall the post-Lie algebra first introduced in \cite{BK}.
\begin{defi} {\rm(\cite[Definition 4.1 and Definition 4.2]{BK})}
We define the Lie bracket on $\mathcal{V}=\tilde{\mathcal{V}}\oplus\mathbb{R}X_0\oplus\ldots\oplus \mathbb{R}X_d$ as $[x, y]_{1} = 0$ for $x, y \in \tilde{\mathcal{V}}$, $[x, y]_{1} = 0$ for $x, y \in \mathbb{R}X_0\oplus\ldots\oplus \mathbb{R}X_d$ and as
	\begin{equation*} \label{Lie-bracket}
		[I_a(\tau),X_i]_{1} =  I_{a-e_i}(\tau).
	\end{equation*}
We define a product $ \widehat{\triangleright} $ on $ \mathcal{V} $ for every $ a,b \in \mathbb{N}^{d+1}, \, i,j \in  \lbrace 0,\ldots,d\rbrace $ as:
	\begin{equation*}
		X_i \, \widehat{\triangleright}  \,  I_{a}(\tau) = I_{a}(  \uparrow^i \tau), \quad I_{a}(\tau) \,  \widehat{\triangleright}  \, X_{i} = 0, \quad  X_i \, \widehat{\triangleright}  \, X_{j} = 0,
	\end{equation*}
	and
	\begin{equation*}
		I_{a}(\sigma) \, \widehat{\triangleright}  \, I_{b}(\tau) = I_{a}(\sigma) \, \widehat{\curvearrowright} \,I_{b}(\tau).
	\end{equation*}
	
\end{defi}

\begin{thm} {\rm(\cite[Theorem 4.4]{BK})}\label{pre-lie-def-to-post}
	The triple $(\mathcal{V}, [\cdot,\cdot]_{1}, \widehat{\triangleright})$ is a post-Lie algebra.
\end{thm}

\begin{rmk}
By the Guin-Oudom functor, the universal enveloping
algebra of the post-Lie algebra $(\mathcal{V}, [\cdot,\cdot]_{1}, \widehat{\triangleright})$ is a post-Hopf algebra $(U(\mathcal{V}),\cdot,\widehat{\triangleright})$ \cite{LST-JNCG}. Moreover, the graded dual Hopf algebra of the sub-adjacent Hopf algebra $(U(\mathcal{V}),*)$ has been used in regularity structures \cite{BHZ,Hairer}. 
\end{rmk}

Now, we want to see the triple $( \mathcal{V}, [\cdot,\cdot]_{1}, \widehat{\triangleright})$ as a post-Lie deformation of a pre-Lie algebra $( \mathcal{V},\triangleright)$ where the  product $ \triangleright $ is given by
\begin{equation} \label{axiom_post}
	X_i \, \triangleright  \,  I_{a}(\tau) =  I_{a}(\uparrow^i \tau), \quad I_{a}(\tau) \,  \triangleright  \, X_{i} = 0, \quad  X_i \, \triangleright \, X_{j} = 0,
\end{equation}
and
\begin{equation*}
	I_{a}(\sigma) \, \triangleright \, I_{b}(\tau) = I_{a}(\sigma) \, \curvearrowright \,I_{b}(\tau).
\end{equation*}

\begin{pro}\label{derivation-pre-lie}
With the above notations, $( \mathcal{V},\triangleright)$ is a pre-Lie algebra.
\end{pro}
\begin{proof} One has to check
	\begin{equation*}
		\left( \tau_1  \triangleright \tau_2 \right)  \triangleright \tau_3 - 	 \tau_1  \triangleright (  \tau_2  \triangleright\tau_3 ) = 	( \tau_2  \triangleright \tau_1 )  \triangleright \tau_3 - 	 \tau_2  \triangleright (  \tau_1  \triangleright \tau_3 )
	\end{equation*}
	where the $ \tau_i $ are of the form $ X_i  $ and $ I_a(\sigma) $. We have to look at the different cases.

		(i) If the $ \tau_i $ are all planted decorated trees then the identity follows from the fact that the planted decorated equipped with the grafting operator $  \curvearrowright$ is a pre-Lie algebra. 

		(ii) If $\tau_3$ is of the form $X_i$, the identity is trivially equal to zero.

		(iii) It remains two cases to check for the form of $ (\tau_1,\tau_2,\tau_3):$
		\begin{equation*}
			(X_i,X_j,I_a(\sigma)), \quad (X_i,I_b(\tau),I_a(\sigma)).
		\end{equation*}
		We omit $ (I_b(\tau),X_i,I_a(\sigma)) $ as it gives the same identity as $ (X_i,I_b(\tau),I_a(\sigma)) $.
		Let us start with $ (X_i,X_j,I_a(\sigma)) $, one has
		\begin{equation*}
			\begin{aligned}
			\left( \tau_1  \triangleright \tau_2 \right)  \triangleright \tau_3 - 	 \tau_1  \triangleright (  \tau_2  \triangleright \tau_3 ) & = \left( X_i  \triangleright X_j \right)  \triangleright  I_a(\sigma) - 	 X_i  \triangleright (  X_j  \triangleright I_a(\sigma) )
			\\
			& =  - 	 X_i  \triangleright (  X_j  \triangleright I_a(\sigma) )
			\\ & = -   I_a(\uparrow^i \uparrow^j\sigma).
			\end{aligned}
		\end{equation*}
		On the other hand, one has
		\begin{equation*}
			\begin{aligned}
			( \tau_2  \triangleright \tau_1 )  \triangleright \tau_3 - 	 \tau_2  \triangleright (  \tau_1  \triangleright \tau_3 ) & =
			( X_j  \triangleright X_i )  \triangleright  I_a(\sigma) - 	 X_j \triangleright (  X_i  \triangleright I_a(\sigma)  )
			\\ & = - 	 X_j \triangleright (  X_i  \triangleright I_a(\sigma)  )
			\\&  =  -  I_a(\uparrow^j \uparrow^i	 \sigma).
			\end{aligned}
		\end{equation*}
		We conclude from the fact that $  \uparrow^i \uparrow^j =   \uparrow^j \uparrow^i$.
		We continue with the case $ (X_i,I_b(\tau),I_a(\sigma))$.
		One has
		\begin{equation*}
			\begin{aligned}
			\left( \tau_1  \triangleright \tau_2 \right)  \triangleright \tau_3 - 	 \tau_1  \triangleright (  \tau_2  \triangleright \tau_3 ) & = \left( X_i  \triangleright I_b(\tau) \right)  \triangleright  I_a(\sigma) - 	 X_i  \triangleright ( I_b(\tau) \triangleright I_a(\sigma) )
			\\
			& =  - 	 I_b(\tau) \triangleright (  X_i  \triangleright I_a(\sigma) ),
			\end{aligned}
		\end{equation*}
		where the derivation identity \eqref{derivation_i} for $ \uparrow^i $ is used. On the other hand, one has
		\begin{equation*}
			\begin{aligned}
			( \tau_2  \triangleright \tau_1 )  \triangleright \tau_3 - 	 \tau_2  \triangleright (  \tau_1  \triangleright \tau_3 ) & =
			(  I_b(\tau)   \triangleright X_i )  \triangleright  I_a(\sigma) - 	  I_b(\tau)  \triangleright (  X_i  \triangleright I_a(\sigma)  )
			\\ & = - I_b(\tau)  \triangleright (  X_i  \triangleright I_a(\sigma)  ),
			\end{aligned}	
		\end{equation*}
		which allows to conclude. The proof is finished.
\end{proof}

\begin{rmk}
	The previous pre-Lie structure could be seen as a way to define a new pre-Lie algebra from a pre-Lie product and a derivation.
	Indeed, we suppose given a pre-Lie algebra $(\mathcal{V}, \curvearrowright  )$ and  commutative derivations $ \uparrow^i $ indexed by $\left\lbrace 0,...,d \right\rbrace $ for this pre-Lie algebra. One has for every $ x,y \in \mathcal{V} $
	\begin{equation*}
		\uparrow^i \left( x \curvearrowright y  \right) = \left(  \uparrow^i  x \right) \curvearrowright y  + x \curvearrowright \uparrow^i \left( y  \right).
	\end{equation*}
	Then, we extend the space $ \mathcal{V} $ into $ \tilde{\mathcal{V}} $ by adding new elements $ X_i $. We can then equip   $ \tilde{\mathcal{V}} $ with a pre-Lie product $ \tilde{\curvearrowright} $ viewed as an extension of  $ \curvearrowright $ to $ \tilde{\mathcal{V}} $:
	\begin{equation*}
		\begin{aligned}
		& 	x \, \tilde{\curvearrowright} \, y  = x \curvearrowright y, \quad X_i \, \tilde{\curvearrowright} \, x = \uparrow^i x,
		\\ & X_i \, \tilde{\curvearrowright} \, X_j = 0, \quad  x \, \tilde{\curvearrowright} \, X_i = 0.
		\end{aligned}
	\end{equation*}
	for every $ x,y \in \tilde{\mathcal{V}} $. The proof follows exactly the same arguments as above which are the fact that $ \curvearrowright $ is a pre-Lie product and the $ \uparrow^i $ are commutative derivations for this product.
\end{rmk}

\emptycomment{
\tr{This observation is very interesting.  It is a canonical way to construct pre-Lie algebras from  commutative derivations. This construction is called derivation extensions of pre-Lie algebras.}
{\color{red} Do you have a reference?}
\TR{Let $(\g,\rhd)$ be a pre-Lie algebra and $\{d_i\}_{i=1}^{n}$ be commutative derivations of the pre-Lie algebra $(\g,\rhd)$. We consider the vector space $\h=:\mathbb{R}d_1\oplus\ldots\oplus \mathbb{R}d_n$ which is generated by $\{d_i\}_{i=1}^{n}$. Moreover, we regard the vector space $\h$ as the trivial pre-Lie algebra. Since $\{d_i\}_{i=1}^{n}$ are commutative derivations of the pre-Lie algebra $\g$,
there is a bimodule structure of the trivial pre-Lie algebra $\h$ on the vector space $\g$, which is defined by
\begin{eqnarray*}
d_i\cdot x=d_i(x),\,\,x\cdot d_i=0,\,\,\forall x\in\g,~i=1,\cdots,n.
\end{eqnarray*}
On the other hand, the bimodule structure of the pre-Lie algebra $(\g,\rhd)$ on the vector space $\h$ is trivial, which means that
$x\cdot h=0,\,\,h\cdot x=0,~x\in\g,~h\in\h.$ By a simple computation, we obtain that $(\g,\h)$ is a matched pair of pre-Lie algebras \cite[Theorem 3.5]{Bai-1}.
So there is a pre-Lie algebra structure on the vector space $\g\oplus\h$, which is given by
\begin{eqnarray*}
&&x\rhd y:=x\rhd y,~h\rhd x:=h(x),\\
&&x\rhd h:=0,\,\,\,\,\,\,\,\,\,\,\,~h\rhd l:=0.
\end{eqnarray*}
Here $x,y$ are elements of $\g$ and $h,l$ are elements of $\h$.
}
}

\begin{thm}\label{example-def}
The post-Lie algebra $(\mathcal{V}, [\cdot,\cdot]_{1}, \widehat{\triangleright})$ in Theorem  \ref{pre-lie-def-to-post} is a post-Lie deformation of the pre-Lie algebra $(\mathcal{V},[\cdot,\cdot]_0=0,\triangleright)$ in Proposition \ref{derivation-pre-lie}.
\end{thm}
\begin{proof}
One can analyse the post-Lie algebraic structure $(\mathcal{V}, [\cdot,\cdot]_{1}, \widehat{\triangleright})$ in the following way:
\begin{equation*}
	\begin{aligned}
		[I_a(\tau),X_i]_{1} & = [I_a(\tau),X_i]_0 +  \hbox{lower grading terms},\\
	I_{a}(\sigma) \, \widehat{\triangleright}  \, I_{b}(\tau) & = I_{a}(\sigma) \, \triangleright \,I_{b}(\tau)  +
	\hbox{lower grading terms}.
	\end{aligned}
\end{equation*}
 More precisely, we define the linear map $\pi:\wedge^2\mathcal{V}\lon \mathcal{V}$ by
 \begin{eqnarray}\label{deformation-post-1}
	\pi(u,v):=\left\{
	\begin{array}{ll}
		I_{a-e_i}(\tau), &\mbox {$u=I_a(\tau),v=X_i$,}\\
          -I_{a-e_i}(\tau), &\mbox {$u=X_i,v=I_a(\tau)$,}\\
		0, &\mbox {other case.}
	\end{array}
	\right.
\end{eqnarray}
And define the linear map $\omega:\otimes^2\mathcal{V}\lon \mathcal{V}$ by
 \begin{eqnarray}\label{deformation-post-2}
	\omega(u,v):=\left\{
	\begin{array}{ll}
		\sum_{v\in N_{\tau}}\sum_{|\ell| \neq 0}{\mathfrak{n}_v \choose \ell} I_b( \sigma  \curvearrowright_v^{a-\ell}(\uparrow_v^{-\ell} \tau)), &\mbox {$u=I_a(\sigma),v=I_b(\tau)$,}\\
		0, &\mbox {other case.}
	\end{array}
	\right.
\end{eqnarray}
By direct computation, we obtain $[u,v]_1=[u,v]_0+\pi(u,v)$ and $u\,\widehat{\triangleright}\,v=u\,\triangleright\,v+\omega(u,v)$. Therefore, the post-Lie algebra $(\mathcal{V}, [\cdot,\cdot]_{1}, \widehat{\triangleright})$ is a post-Lie deformation of the pre-Lie algebra $( \mathcal{V},[\cdot,\cdot]_0=0,\triangleright)$.
\end{proof}

\begin{rmk}\label{parameter-def}
 Let $t\in \mathbb{R}$ be a real parameter. Define the Lie bracket on $\mathcal{V}=\tilde{\mathcal{V}}\oplus\mathbb{R}X_0\oplus\ldots\oplus \mathbb{R}X_d$ as $[x, y]_{1,t} = 0$ for $x, y \in \tilde{\mathcal{V}}$, $[x, y]_{1,t} = 0$ for $x, y \in \mathbb{R}X_0\oplus\ldots\oplus \mathbb{R}X_d$ and as
	\begin{equation*}
		[I_a(\tau),X_i]_{1,t}    = [I_a(\tau),X_i]_0 + t I_{a-e_i}(\tau).
	\end{equation*}
We define a product $ \widehat{\triangleright} $ on $ \mathcal{V} $ for every $ a,b \in \mathbb{N}^{d+1}, \, i,j \in  \lbrace 0,\ldots,d\rbrace $ as:
	\begin{equation*}
		X_i \, \widehat{\triangleright_t}  \,  I_{a}(\tau) = I_{a}(  \uparrow^i \tau), \quad I_{a}(\tau) \,  \widehat{\triangleright_t}  \, X_{i} = 0, \quad  X_i \, \widehat{\triangleright_t}  \, X_{j} = 0,
	\end{equation*}
	and
	\begin{equation*}
		I_{a}(\sigma) \, \widehat{\triangleright}_t  \, I_{b}(\tau)
		= I_{a}(\sigma) \, \triangleright \, I_{b}(\tau)+	\sum_{v\in N_{\tau}}\sum_{|\ell| \neq 0}{\mathfrak{n}_v \choose \ell} t^{|\ell|}	 I_b( \sigma  \curvearrowright_v^{a-\ell}(\uparrow_v^{-\ell} \tau)).
	\end{equation*}
By a similar computation in Theorem \ref{pre-lie-def-to-post},  for any $t\in \mathbb{R}$, the triple $( \mathcal{V}, [\cdot,\cdot]_{1,t}, \widehat{\triangleright_t})$ is a post-Lie algebra. Moreover, $(\mathcal{V}, [\cdot,\cdot]_{1,t}, \widehat{\triangleright_t})$ is a post-Lie deformation of the pre-Lie algebra $( \mathcal{V},[\cdot,\cdot]_0=0,\triangleright)$.
\end{rmk}

Inspired by the above remark, we can construct a formal post-Lie deformation as following:  for $i=1,\cdots,$ we define the linear maps $\pi_i:\wedge^2\mathcal{V}\lon \mathcal{V}$,  by
$$
	\pi_1(u,v):=\pi(u,v),\,\,\pi_{2}=\cdots=\pi_{n}=\cdots=0;
$$
for $i=1,\cdots,$ we define the linear maps $\omega_i:\otimes^2\mathcal{V}\lon \mathcal{V}$ by
 \begin{eqnarray*}
	\omega_i(u,v):=\left\{
	\begin{array}{ll}
		\sum_{v\in N_{\tau}}\sum_{|\ell|=i}{\mathfrak{n}_v \choose \ell} I_b( \sigma  \curvearrowright_v^{a-\ell}(\uparrow_v^{-\ell} \tau)), &\mbox {$u=I_a(\sigma),v=I_b(\tau)$,}\\
		0, &\mbox {other case.}
	\end{array}
	\right.
\end{eqnarray*}
Now we consider the $t$ as a formal variable.  Further, we consider the formal power series
\begin{eqnarray}
\pi_t=\sum_{i=0}^{+\infty}\pi_i t^i=\pi_1 t,\,\,\,\,~~~
\omega_t=\sum_{i=0}^{+\infty}\omega_i t^i,
\end{eqnarray}
here $\pi_0=0$ and $\omega_0=\triangleright$.
\begin{thm}
With the above notations, $(\mathcal{V}[[t]],\pi_t,\omega_t)$ is an $\mathbb{R}[[t]]$-post-Lie algebra and  $(\pi_t,\omega_t)$ is a formal post-Lie deformation of the pre-Lie algebra $(\mathcal{V},[\cdot,\cdot]_0=0,\triangleright)$ in Proposition \ref{derivation-pre-lie}.
\end{thm}
\begin{proof}
We   only need to prove that for any $n$ and $x,y,z\in \mathcal{V}$,   the equations \eqref{deformation4}, \eqref{deformation5} and \eqref{deformation6} hold.
Since $\pi_1$ is a Lie algebra structure on $\mathcal{V}$ and $\pi_{i}=0,i\not=1$, the equation \eqref{deformation4} holds  immediately. One can show the equations \eqref{deformation5} and \eqref{deformation6} directly by a tedious computation. Here we take a different approach.
We suppose that the equation \eqref{deformation6}  does not hold. This  means that there exist $n_0$ and $x_0,y_0,z_0\in \mathcal{V}$ such that
$$
\sum\limits_{i+j=n_0\atop i,j\geq0}\Big(\omega_i(\omega_j(x_0,y_0)-\omega_j(y_0,x_0)+\pi_j(x_0,y_0),z_0)-\omega_i(x_0,\omega_j(y_0,z_0))+\omega_i(y_0,\omega_j(x_0,z_0))\Big)\not=0.
$$
By the definition of $\omega_i$, we obtain a nonzero polynomial
\begin{eqnarray*}
&&\omega_t(\omega_t(x_0,y_0)-\omega_t(y_0,x_0)+\pi_t(x_0,y_0),z_0)-\omega_t(x_0,\omega_t(y_0,z_0))+\omega_t(y_0,\omega_t(x_0,z_0))\\
&=&\sum_{n=0}^{+\infty}\sum\limits_{i+j=n\atop i,j\geq0}\Big(\omega_i(\omega_j(x_0,y_0)-\omega_j(y_0,x_0)+\pi_j(x_0,y_0),z_0)-\omega_i(x_0,\omega_j(y_0,z_0))+\omega_i(y_0,\omega_j(x_0,z_0))\Big)t^n.
\end{eqnarray*}
According to the fundamental theorem of algebra, the above nonzero polynomial has finite many real roots. Note that  for any $t\in \mathbb{R}$,  $( \mathcal{V}, [\cdot,\cdot]_{1,t}, \widehat{\triangleright_t})$ is a post-Lie algebra, it means that the above nonzero polynomial has infinite many real roots. Therefore, our assumption gives rise to a contradiction. So the equation \eqref{deformation6} holds. By the same method, we deduce that the equation \eqref{deformation5} holds. Now the  whole process of the proof is finished.
\end{proof}

\begin{rmk}
One can observe that it is a specific case of formal post-Lie deformation of the pre-Lie algebra $(\mathcal{V},[\cdot,\cdot]_0=0,\triangleright)$ as $ \omega_n(	I_{a}(\sigma),I_{b}(\tau))$ is zero for $n$ sufficiently large. If we take $t=1$, then we obtain the previous post-Lie deformation of the pre-Lie algebra  $(\mathcal{V},[\cdot,\cdot]_0=0,\triangleright)$ given in Theorem \ref{example-def}.
\end{rmk}

	\vspace{2mm}
	\noindent
	{\bf Acknowledgements. }  The first author (Y. Bruned) gratefully acknowledge funding support from the European Research Council (ERC) through the ERC Starting Grant Low Regularity Dynamics via Decorated Trees (LoRDeT), grant agreement No.\ 101075208. The second and third authors (Y. Sheng and R. Tang) were partially supported by NSFC (12471060, W2412041, 12371029) and the Fundamental Research Funds for the Central Universities.


\begin{thebibliography}{a}
\bibitem{AFM} M. Al-Kaabi, K. Ebrahimi-Fard and D. Manchon, Free post-groups, post-groups from group actions, and post-Lie algebras. \emph{J. Geom. Phys.}	{\bf198} (2024),  Paper No. 105129, 15 pp.


\bibitem{Bai} C. Bai, An introduction to pre-Lie algebras. In: Algebra and Applications 1: Nonssociative Algebras and Categories, Wiley Online Library (2021), 245-273.



		\bibitem{BGST}
		C. Bai, L. Guo, Y. Sheng  and R. Tang, Post-groups, (Lie-)Butcher groups and the Yang-Baxter equation. \emph{Math. Ann.} {\bf388} (2024), 3127-3167.
					
		\bibitem{BGN}
		C. Bai, L. Guo and X. Ni, Nonabelian generalized Lax pairs, the classical Yang-Baxter equation and PostLie algebras. \emph{Comm. Math. Phys.} {\bf 297} (2010), 553-596.

   \bibitem{Bal} D. Balavoine, Deformations of algebras over a quadratic operad. Operads: Proceedings of Renaissance Conferences (Hartford, CT/Luminy, 1995), \emph{Contemp. Math.} {\bf 202} Amer. Math. Soc., Providence, RI, 1997, 207-34.
			
		\bibitem{BHZ}
		Y. Bruned, M. Hairer and L. Zambotti, Algebraic renormalisation of regularity structures. \emph{Invent. Math.} {\bf 215} (2019),  1039-1156.

        \bibitem{BCCH}
{ \rm Y. Bruned, A. Chandra, I. Chevyrev and M. Hairer}, \newblock  Renormalising SPDEs in regularity structures. \emph{J. Eur. Math. Soc. (JEMS)}, \textbf{23} (2021), 869-947.
		
		\bibitem{BK}
		Y. Bruned and F. Katsetsiadis, Post-Lie algebras in regularity structures. 	\emph{Forum Math. Sigma} {\bf11} (2023), 1-20.
		
		\bibitem{BM22}
		Y.~Bruned and D.~Manchon, \newblock Algebraic deformation for (S)PDEs. \emph{J. Math. Soc. Japan.} \textbf{75} (2023), 485-526.
		
		\bibitem{Boyom}
		M. Boyom, The cohomology of Koszul-Vinberg algebras. \emph{Pacific J. Math.} {\bf225} (2006), 119-153.

\bibitem{Bu0} D. Burde, Simple left-symmetric algebras with solvable Lie algebra. \emph{Manuscripta Math.} {\bf 95} (1998), 397-411.

\bibitem{Bu}
		D. Burde, Left-symmetric algebras, or pre-Lie algebras in geometry and physics. {\em Cent. Eur. J. Math.} {\bf 4} (2006), 323-357.
		

\bibitem{Burmester-1}
A. Burmester and U. K\"uhn, On post-Lie structures for free Lie algebras. arXiv:2504.19661.

\bibitem{Burmester-2}
A. Burmester, N. Confurius and U. K\"uhn, AGZT-Lectures on formal multiple zeta values. arXiv:2406.13630.
		
		
						
		\bibitem{CL}
		F. Chapoton and M. Livernet, Pre-Lie algebras and the rooted trees operad. {\em  Int. Math. Res. Not.}  {\bf 8} (2001),  395-408.
		
		
		
		
		\bibitem{D}
		V. Dotsenko,  Functorial PBW theorems for post-Lie algebras. \emph{Comm. Algebra} {\bf48} (2020),  2072-2080.

\bibitem{DA}
		A. Dzhumadil'daev, Cohomologies and deformations of right-symmetric algebras. \emph{J. Math. Sci. (New York)}  {\bf 93} (1999),  836-876.
		
		\bibitem{EMM}
		K. Ebrahimi-Fard, I. Mencattini and H. Z. Munthe-Kaas, Post-Lie algebras and factorization theorems. {\em J. Geom. Phys.} {\bf 119} (2017), 19-33.



	\bibitem{F2018}
	L.~Foissy, 	Algebraic structures on typed decorated rooted trees. \emph{SIGMA} \textbf{17} (2021), 1-28.
				
		
		
		\bibitem{Ge0}
		M. Gerstenhaber, The cohomology structure of an associative ring. \emph{Ann. of Math. (2)} {\bf 78} (1963), 267-288.
		
		\bibitem{Ge}
		M. Gerstenhaber, On the deformation of rings and algebras. \emph{Ann. of Math. (2)} {\bf 79} (1964), 59-103.
		
		

	\bibitem{Gubarev}
		V. Gubarev, Poincar\'e-Birkhoff-Witt theorem for pre-Lie and post-Lie algebras. \emph{J. Lie Theory} {\bf30} (2020),  223-238.		


\bibitem{GMS}
E. Grong, H. Z. Munthe-Kaas  and J. Stava,  Post-Lie algebra structure of manifolds with constant curvature and torsion. \emph{J. Lie Theory} {\bf34} (2024),  339-352.



\bibitem{Hairer}
		M. Hairer,  A theory of regularity structures. \emph{Invent. Math.} {\bf198} (2014),  269-504.

\bibitem{Ha} R. Hartshore, Deformation Theory. \emph{Graduate Texts in Math} {\bf 257}, Springer, 2010.
				
		
\bibitem{JZ}
J. Jacques and  L. Zambotti, Post-Lie algebras of derivations and regularity structures. arXiv:2306.02484.

\bibitem{KS}
K. Kodaira and D. Spencer, On deformations  of complex analytic structures I \& II. {\em Ann. Math.} {\bf 67} (1958), 328-466.

				
\bibitem{LST24} A. Lazarev, Y. Sheng and R. Tang, Homotopy theory of post-Lie algebras.  arXiv:2504.19998.

\bibitem{LST-JNCG}
Y. Li, Y. Sheng and R. Tang, Post-Hopf algebras, relative Rota-Baxter operators and solutions to the Yang-Baxter equation. \emph{J. Noncommut. Geom.} {\bf18} (2024), 605-630.
		
		
		
\bibitem{Ma} D. Manchon, A short survey on pre-Lie algebras. In: Noncommutative Geometry and Physics: Renormalisation, Motives, Index Theory (2011), 89-102.

\bibitem{Maz} B. Mazur, Perturbations, deformations, and variations (and ``near-misses") in geometry, physics, and number theory. \emph{Bull AMS} {\bf 41} (2004), 307-336.
		
	\bibitem{MQS} I. Mencattini, A. Quesney and P. Silva,  Post-symmetric braces and integration of post-Lie algebras. \emph{J. Algebra} {\bf556} (2020), 547-580.
		


	
		\bibitem{Munthe-Kaas-Lundervold}
		H. Munthe-Kaas and A. Lundervold, On post-Lie algebras, Lie-Butcher series and moving frames. \emph{Found. Comput. Math.} {\bf 13} (2013), 583-613.


        \bibitem{NR} A. Nijenhuis  and R. Richardson,  Cohomology and deformations in graded Lie algebras. {\em Bull. Amer. Math. Soc.} {\bf 72} (1966) 1-29.

		
		

		\bibitem{Sm22b} A. Smoktunowicz,  On the passage from finite braces to pre-Lie rings. {\em Adv. Math.} {\bf409} (2022), Paper No. 108683, 33 pp.
			
		\bibitem{Val} B. Vallette, Homology of generalized partition posets. {\em J. Pure Appl. Algebra} {\bf 208} (2007), 699-725.
		
	\end{thebibliography}
\end{document}